\documentclass[journal,12pt]{article}

\usepackage{subfigure}
\usepackage{amssymb,amsbsy,amsmath,float,amsthm}
\usepackage{paralist}
\usepackage{graphics} 
\usepackage{epsfig} 
\usepackage{graphicx}
\usepackage{epstopdf}
\usepackage[colorlinks=true]{hyperref}

\hypersetup{urlcolor=blue, citecolor=red}

\newtheorem{theorem}{Theorem}[section]

\newtheorem{proposition}{Proposition} [section]

\newtheorem{remark}{Remark}
\begin{document}
\onecolumn
\title{Image representation and denoising using squared eigenfunctions of Schr\"odinger operator }
\date{}

\maketitle

\centerline{Zineb KAISSERLI$^{*,**,1}$, Taous-Meriem LALEG-KIRATI$^{*,2}$} 
\medskip

{\footnotesize
 \centerline{ $^{*}$Computer, Electrical and Mathematical Science and Engineering Division}
 \centerline{King Abdullah University of Science and Technology (KAUST), KSA}
 \centerline{and}
 \centerline{$^{**}$Mathematical and Computer Science Division}
 \centerline{Abdelhamid Ibn Badis University (UMAB), Algeria}
} 
\medskip
{\footnotesize
 \centerline{Emails : $^{1}$kaisserli.z@gmail.com and $^{2}$taousmeriem.laleg@kaust.edu.sa}}

\bigskip

\begin{abstract}
This paper extends to two dimensions the recent signal analysis method based on the semi-classical analysis of the Schr\"odinger operator. The generalization uses the separation of variables technique when writing the eigenfunctions of the Schr\"odinger operator. The algorithm is described and the effect of some parameters on the convergence of this method are numerically studied. Some examples on image reconstruction and denosing are illustrated.\\

\end{abstract}


\section{Introduction}
A new signal analysis method has been proposed in \cite{Laleg-Kirati2013}. The idea consists is decomposing the signal using a family of a spatially shifted and localized functions, which are given by the squared $L^2$-normalized eigenfunctions associated to the discrete spectrum of the one dimensional semi-classical Schr\"odinger operator, where the signal is considered as a potential of this operator. It has been shown in \cite{Helffer2011, Laleg-Kirati2013} that the discrete spectrum consisting of negative eigenvalues and the associated squared $L^2$-normalized eigenfunctions can be used to reconstruct, estimate and analyze the signal. This method has been denoted in \cite{Laleg-Kirati2013} SCSA for \textbf{S}emi-\textbf{C}lassical \textbf{S}ignal \textbf{A}nalysis. Besides its interesting localization property, the SCSA method has proved its performance in some applications. For instance, interesting results have been obtained when applying the SCSA method to the analysis of arterial blood pressure signals \cite{Laleg-Kirati2007, Laleg-Kirati2010, Laleg-Kirati2013} and to the analysis of the performance of turbomachinery \cite{Eleiwi2011}. Moreover, it has been shown in \cite{Liu2012}, that the SCSA method can cope with noisy signals, making this method a potential tool for denoising. The filtering property of the SCSA method is currently under study  through in-vivo experiments with Magnetic Resonance Spectroscopy data \cite{Laleg-KiratiSeptember9-102014}.  

In this paper, the SCSA method is extended to two dimensions (2D). This extension is useful for image representation and denoising. The considered approach consists in using separation of variables method when writing the squared $L^2$-normalized eigenfunctions of a 2D semi-classical Schr\"odinger operator, where the image is considered as a potential of this operator.  The problem consists then in solving the spectral problem for 1D Schr\"odinger operators and in combining the results for the reconstruction, estimation and  denoising of images with an appropriate formula inspired from semi-classical analysis theory \cite{Helffer1990, Karadzhov1986}. This formula can be written as the product of the squared $L^2$-normalized eigenfunctions in both directions. The idea of using 1D transforms for 2D reconstruction is often used in image processing \cite{Dudgeon1993,Jain1989}, the 2D Fourier transform is an example \cite{Mallat2009}. As in the 1D case, the convergence of this formula when the semi-classical parameter converges to zero is proved for image reconstruction. We will also show that this method can be used for image denoising and we will illustrate the results through some examples.

In section \ref{sec2}, the 1D SCSA method is described. Then, in section \ref{sec3}, the 2D SCSA formula is presented followed by the convergence analysis when the semi-classical parameter goes to $0$. In section \ref{sec4}, an algorithm based on the spectral problems of 1D Schr\"odinger operators and tensor product is introduced. The analysis of some parameters and the use of this algorithm for image representation is illustrated in section \ref{sec5}. First results on image denoising based on 2D SCSA approach are presented in section \ref{sec6}. Finally the last section summarizes and discusses the obtained results.

\section{Preliminary (SCSA in $1$D case)}\label{sec2}
In this section, we recall the idea behind the SCSA method \cite{Helffer2011,Laleg-Kirati2013}. Let us consider the following one dimensional semi-classical Schr\"odinger operator:
\begin{equation}\label{Schrodinger}
    \mathcal{H}_{1,h}(V_{1})\psi= -h^{2}\frac{d^{2}\psi}{dx^{2}}-V_{1}\psi,\quad \psi \in \mathbf{H}^2(\mathbb{R}),
\end{equation}
where $h\in \mathbb{R}^{*}_{+}$ is the semi-classical parameter \cite{Dimassi1999}, and $V_{1}$ is a positive real valued function belonging to  $\mathcal{C}^{\infty}(\Omega_{1})$ where  $\Omega_{1}\subset \mathbb{R}$ is compact. Here $\mathbf{H}^2(\mathbb{R})$ denotes the Sobolev space of order $2$. Then, the potential $V_{1}$ can be represented using the following proposition.

\begin{proposition} $\cite{Helffer2011}$ \label{prop1}
Let $V_{1} \in  \mathcal{C}^{\infty}(\Omega_{1})$ be positive real valued function, where $\Omega_{1}\subset \mathbb{R}$ is compact. Then, $V_{1}$ can be represented using the following formula:
\begin{equation}\label{SCSAhg}
    V_{1,h,\gamma,\lambda}(x)=-\lambda+\left(\frac{h}{L^{cl}_{1,\gamma}}\sum_{k=1}^{K_h^{\lambda}}
    (\lambda-\mu_{k,h})^{\gamma}\psi_{k,h}^{2}(x)\right)^{\frac{2}{1+2\gamma}},
\end{equation}
where $h\in \mathbb{R}^{*}_{+}$, $\gamma\in \mathbb{R}_{+}$, $ \lambda \in \mathbb{R}^{*}_{-}$, and $ L^{cl}_{1,\gamma} $ is the suitable universal semi-classical constant given by: 
\begin{equation*}
L^{cl}_{1,\gamma}=\frac{1}{2\sqrt{\pi}}\frac{\Gamma(\gamma+1)}{\Gamma(\gamma+\frac{3}{2})},
\end{equation*}
where $\Gamma$ is the Gamma function.

Moreover, $\mu_{k,h}$ are the negative eigenvalues of the operator $\mathcal{H}_{1,h}(V_{1})$ with  $\mu_{1,h} < \cdots <\mu_{{K_h^{\lambda}},h}< \lambda$, $K_{h}^{\lambda}$ is the number of negative eigenvalues smaller than $\lambda$, and $\psi_{k,h}$ are the associated $L^2$-normalized eigenfunctions such that:
\begin{equation*}
    \mathcal{H}_{1,h}(V_{1})\,\psi_{k,h}=\mu_{k,h}\psi_{k,h}, \quad k=1,\cdots,K_h^{\lambda}
\end{equation*}
\end{proposition}

If a signal is interpreted as a potential $V_1$ of the Schr\"odinger operator,  then the formula $(\ref{SCSAhg})$ can be used for signal analysis and reconstruction. Indeed
the efficiency of the proposed signal estimation method and the influence of the parameters $\lambda$, $\gamma$ and $h$ have been studied  in \cite{Helffer2011}. In particular,   as it is described in \cite{Helffer2011} and \cite{Laleg-Kirati2013}, the semi-classical parameter $h$ plays a key role in this approach. In fact, when $h$ decreases, the estimation  $V_{1,h,\gamma,\lambda}$ improves. Since the study of the Schr\"{o}dinger operator in the case where $h$ tends to $0$ is referred to the semi-classical
analysis \cite{Dimassi1999}, this justifies the name \textbf{S}emi-\textbf{C}lassical \textbf{S}ignal \textbf{A}nalysis that we give to this method \cite{Helffer2011,Laleg-Kirati2013}.

Let us point out that the formula given in (\ref{SCSAhg}) is still valid in the case where $\lambda=0$ and it gives good results. For exemple, the following formula:
\begin{equation}
    V_{1,h,\frac{1}{2},0}(x)=4h\sum_{k=1}^{K_h^{0}}(-\mu_{k,h})^{\frac{1}{2}}\psi_{k,h}^{2}(x),
\end{equation}
was successfully used in the analysis of arterial blood pressure signal in \cite{Laleg-Kirati2007, Laleg-Kirati2010,Laleg-Kirati2013}, and the analysis of the performance of turbo machines in \cite{Eleiwi2011}.


\section{A Two-dimensional Schr\"odinger operator: Asymptotic results}\label{sec3}

From now on, we consider the following $2$D semi-classical Schr\"odinger operator associated to a potential $V_2$:
\begin{equation}\label{Schro2D}
     \mathcal{H}_{2,h}(V_{2})\psi= -h^{2}\Delta\psi-V_{2}\psi, \quad \psi \in \mathbf{H}^2(\mathbb{R}^2),
\end{equation}
where $\Delta:=\frac{\partial^{2}}{\partial x^{2}}+\frac{\partial^{2}}{\partial y^{2}}$ is the $2$D Laplacien operator, $h\in \mathbb{R}^{*}_{+}$ is the semi-classical parameter \cite{Dimassi1999}, and $V_{2}$ is a positive real valued function belonging to  $\mathcal{C}^{\infty}(\Omega_{2})$ where  $\Omega_{2}\subset \mathbb{R}^2$ is compact. $\mathbf{H}^2(\mathbb{R}^2)$ is the sobolev space of order $2$.

Then, inspired from semi-classical properties of the 2D Schr\"odinger operator \cite{Helffer1990}, \cite{Karadzhov1986}, the extension of the SCSA formula to the  $2$D case is given by the following theorem.

\begin{theorem} \label{Main_th}
Let $V_2$  be a positive real valued $\mathcal{C}^{\infty}$ function on a bounded open set $]a,b[\times]c,d[$ considered as potential of Schr\"odinger operator $\left( \ref{Schro2D} \right)$. Then, for any pair $(\Omega_2,\lambda)$ such that $\Omega_2$ is compact and
\begin{equation} \label{ConditionLambda}
\left\{
\begin{array}{l}
\lambda < \inf (V_2 (a,c), V_2(b,d)) \,, \\
V_2(]a,b[\times]c,d[) \subset ]- \lambda, +\infty[\,, \\
-\lambda \mbox{ is not a critical value of } V_2, \text{$($for more details see $\cite{Helffer1990}$$)$}%
\end{array}
\right.
\end{equation}
and, uniformly for $(x,y)\in \Omega_2$, we have
\begin{equation} \label{scsa2d}
V_2(x,y)=-\lambda+ \lim_{h\rightarrow0}  \left(\frac{h^{2}}{L^{cl}_{2,\gamma}} \sum_{k=1}^{K_{h}^{\lambda}} \left(\lambda-\mu_{k,h}\right)^{\gamma}\psi^{2}_{k,h}(x,y)\right)^{\frac{1}{1+\gamma}},
\end{equation}
where $\gamma\in \mathbb{R}_{+}^{*}$, and $L^{cl}_{2,\gamma}$ is the suitable universal semi-classical constant given by
\begin{equation}\label{Lcl}
 L^{cl}_{2,\gamma}=\frac{1}{2^{2}\pi}\frac{\Gamma(\gamma+1)}{\Gamma(\gamma+2)},
\end{equation}
and $\Gamma$ refers to the standard Gamma function.

Moreover, $\mu_{k,h}$ and $\psi_{k,h}$ denote the negative eigenvalues with $\mu_{1,h} < \cdots <\mu_{K_{h}^{\lambda},h}< \lambda$, $K_{h}^{\lambda}$ is a finite number of the negative eigenvalues smaller than $\lambda$, and associated $L^2$-normalized eigenfunctions of the operator $\mathcal{H}_{2,h}(V_{2})$ such that:
\begin{equation}\label{Eq_spec2D}
\mathcal{H}_{2,h}(V_{2})\psi_{k,h} = \mu_{k,h}\psi_{k,h}, \quad k=1,\cdots,K_h^{\lambda}.
\end{equation}
\end{theorem}

We propose to show the convergence of formula (\ref{scsa2d}) when the semi-classical parameter $h$ converges to 0. The proof is a generalization of the one proposed in \cite{Helffer2011}.

The following results are used to prove theorem \ref{Main_th}. The next theorem is a generalization to $2D$, of Theorem 4.1 proposed by by Helffer and Laleg in \cite{Helffer2011} which is a suitable extension of Karadzhov's theorem on the spectral function \cite{Karadzhov1986}.

\begin{theorem}\label{th_HL2D}
Let $V_2$ be a real valued $\mathcal {C}^\infty$ function considered as potential of the Schr\"odinger operator
$\left(\ref{Schro2D}\right)$ on a bounded open set $]a,b[\times]c,d[$. Let $e_h^\gamma$, known as spectral function, be defined by: $\forall\left((x,x'),(y,y')\right)\in\left( ]a,b[\times]c,d[\right)^2$,
\begin{equation}\label{sep1}
e_h^\gamma(\lambda,x,y,x',y')=\sum_{\mu_{k,h}\leq\lambda} \left(\lambda-\mu_{k,h}\right)^\gamma_+\psi_{k,h}(x,y)\psi_{k,h}(x',y'),
\end{equation}
when $h\rightarrow0$. $\mu_{k,h}$ and $\psi_{k,h}$ refer to the decreasing negative eigenvalues less than $\lambda$, and associated $L^2$-normalized eigenfunctions of the operators $\mathcal{H}_{2,h}(V_{2})$ respectively.

Then, for any pair $(\Omega_2,\lambda)$ satisfying $\left(\ref{ConditionLambda}\right)$, we have:
\begin{equation}\label{sep2}
e_h^\gamma(\lambda,x,y,x,y)=(2\pi)^{-2}\left(\lambda+V_2(x,y)\right)^{1+\gamma}_+c_\gamma h^{-2}+\mathcal{O}(h^\gamma),\,\,\,\,h\rightarrow0,
\end{equation}
uniformly in $\Omega_2$, where 
\begin{equation*}
c_\gamma = \int_{\mathbb{R}^2} (1 -\eta^2-\eta^{'2})_+^{\gamma} d\eta\,d\eta' \,,
\end{equation*}
and $(\cdot)_+$ refers to the positive part.
\end{theorem}

\begin{theorem} \label{thConv1}$\cite{Helffer1990}$
Let $V_2$ be a real valued function considered as potential of the Shcr\"odinger operator $\left(\ref{Schro2D}\right)$ belonging to $\mathcal{C}^{\infty}(\mathbb{R}^2)$, with
\begin{equation}
-\infty< \inf\, V_2< \liminf \limits_{\substack{|x| \to +\infty \\ |y| \to +\infty}}V_2.
\end{equation}

Let $\lambda\in ] \inf\, V_2,\liminf \limits_{\substack{|x| \to +\infty \\ |y| \to +\infty}}V_2.[$ and suppose that $-\lambda$ is not a critical value for $V_2$ and $h$ be a semi-classical parameter. We denote by:
\begin{equation}
S_\gamma(h,\lambda)=\sum_{\mu_{k,h}\leq\lambda}(\lambda-\mu_{k,h})^\gamma,\,\,\,\,\gamma\geq0,
\end{equation}
the Riesz means of the decreasing eigenvalues $\mu_{k,h}$ less than $\lambda$ of the Schr\"odinger operators $\mathcal{H}_{2,h}(V_{2})$. Then for $\gamma>0$, we have:
\begin{equation}
S_\gamma(h,\lambda)=\frac{1}{h^2}\left(L_{2,\gamma}^{cl}\int_{\mathbb{R}^2}(\lambda+V_2(x,y))_+^{1+\gamma}dx\,dy+\mathcal{O}(h^{2+\gamma}) \right),\,\,\,\,h\rightarrow0,
\end{equation}
where $(\cdot)_+$ is the positive part and $L_{2,\gamma}^{cl}$, known as the suitable universal semi-classical constant, is given by $\left(\ref{Lcl}\right)$.
\end{theorem}


\begin{proof}{$\mathbf{of\,theorem\, \ref{Main_th}}$}~\newline

We will obtain the proof by using a suitable extension of Karadzhov's theorem \cite{Karadzhov1986} on the spectral function (Theorem \ref{th_HL2D}) and some Riesz means connected to a Lieb-Thirrings conjucture proposed by Helffer and Robert \cite{Helffer1990} (Theorem \ref{thConv1}). 

First, by combining the formulas (\ref{sep1}) and (\ref{sep2}), we find: $\forall (x,y) \in \Omega_2$

\begin{equation}\label{eq1}
\sum_{\mu_{k,h}\leq\lambda}\left(\lambda-\mu_{k,h}\right)_+^{\gamma} \, \psi_{k,h}^2(x,y)= (2\pi h)^{-2}(\lambda+V_2(x,y))^{1+\gamma}_+ c_\gamma + \mathcal{O}(h^{2+\gamma})
\end{equation}
where $h \rightarrow 0$, and $(\cdot)_+$ is the positive part, and
\begin{equation}
c_\gamma = \int_{\mathbb{R}^2} (1 -\eta^2-\eta'^2)_+^{\gamma} d\eta\,d\eta'.
\end{equation}

Now, let's find a simple expression for $c_\gamma$. By integrating the right part of the equation (\ref{eq1}) over $x$ and $y$, we get:

\begin{equation*}
\begin{split}
  \int_{\Omega_2} \left( \sum_{\mu_{k,h}\leq\lambda}\left(\lambda-\mu_{k,h}\right)^\gamma_+\psi_{k,h}^2(x,y)\right)dx\,dy = & \sum_{\mu_{k,h}\leq\lambda}\left(\lambda-\mu_{k,h}\right)^\gamma_+,\,\,\,\, h\rightarrow0\\
  =& \, S_\gamma(h,\lambda),\,\,\,\, h\rightarrow0.
  \end{split}
\end{equation*}

We also have:
\begin{equation*}
  \begin{split}
   & \int_{\Omega_2} h^{-2}\left(\lambda+V_2(x,y)\right)_+^{1+\gamma}(2\pi)^{-2}c_{\gamma}dx\,dy+\mathcal{O}(h^{2+\gamma})\\
   = & \int_{\Omega_2} \sum_{\mu_{k,h}\leq \lambda}\left(\lambda-\mu_{k,h}\right)_+^\gamma\psi_{k,h}^2(x,y)dx\,dy,\,\,\,\, h\rightarrow0\\
   = &\, S_\gamma(h,\lambda),\,\,\,\, h\rightarrow0.
  \end{split}
\end{equation*}

Therefore, for $\gamma>0$, and by using Theorem \ref{thConv1} we obtain:
\begin{equation*}
  \begin{split}
   & \int_{\Omega_2} h^{-2}\left(\lambda+V_2(x,y)\right)_+^{1+\gamma}(2\pi)^{-2}c_{\gamma}dx\,dy+\mathcal{O}(h^{2+\gamma}) = \\
   & \qquad \qquad \qquad \frac{L_{2,\gamma}^{cl}}{h^2}\int_{\Omega_2} \left(\lambda+V_2(x,y)\right)_+^{1+\gamma}dx\,dy+\mathcal{O}(h^{2+\gamma}),\,\,\,\,h\rightarrow0.
  \end{split}
\end{equation*}

Which implies:
\begin{equation*}
\begin{split}
(2\pi)^{-2}c_\gamma &= L^{cl}_{2,\gamma},\\
 &=\frac{1}{2^{2}\pi}\frac{\Gamma(\gamma+1)}{\Gamma(\gamma+2)}.
\end{split}
\end{equation*}

Finally, when the semi-classical parameter $h$ converges to $0$, and as the potential $V_2$ is positive, then the $2$D SCSA formula is given by: $\forall(x,y)\in \Omega_2$

\begin{equation*}
V_2(x,y)=-\lambda+\lim_{h\rightarrow0}\left(\frac{h^{2}} {L^{cl}_{2\gamma}}\sum_{k=1}^{K_{h}^{\lambda}}\left(\lambda-\mu_{k,h}\right)^{\gamma}\psi^{2}_{k,h}(x,y)\right)^{\frac{1}{1+\gamma}}.
\end{equation*}
\end{proof}
\section{New algorithm for image representation based on 2D SCSA formula}\label{sec4}

In image processing, for some geometrical and topological reasons, it is common and more practical to consider a separation of variables approach to extend the 1D transforms to 2D \cite{Dudgeon1993,Jain1989}. This is the case for example with 2D Fourier transform, which can be written using the tensor product of the 1D complex exponential \cite{Mallat2009} or more recently the Ridgelet transform \cite{Do2003} based on the tensor product of 1D wavelet transform. The separation of variables principle allows the design of efficient and fast algorithms where the representation of the image is done row by row and column by column respectively.

The reconstruction of an image using formula $(\ref{scsa2d})$ requires the computation of eigenvalues and eigenfunctions in 2D which is known to be complex and time consuming. Therefore for sake of simplicity, we propose, in this section, to use the separation of variables principle by splitting the 2D operator into two 1D operators and to solve the eigenvalues problems for these 1D operators.

\subsection{Principle in continuous case}

Let us define, for $(x_0, y_0) \in \Omega_2$ the following 1D operators,
\begin{equation}\label{Schro2Dy0}
   		\mathcal{A}_{x_0, h}(V_{2}(x_0,y))\varphi_{x_0}(y) = -h^{2}\frac{\partial^2 \varphi_{x_0}(y)}{\partial y^2} - \frac{1}{2}V_{2}(x_0,y)\varphi_{x_0}(y),
	\end{equation}
\begin{equation}\label{Schro2Dx0}
   		\mathcal{B}_{y_0,h}(V_{2}(x,y_0))\phi_{y_0}(x)= -h^{2}\frac{\partial^2\phi_{y_0}(x)}{\partial x^2} - \frac{1}{2} V_{2}(x,y_0)\phi_{y_0}(x),
	\end{equation}
such that at fixed $(x,y) = (x_0,y_0)$, the summation of the operators $\mathcal{A}_{x_0,h}$ and
$ \mathcal{B}_{y_0,h}$ gives the 2D Schr\"odinger operator $\mathcal{H}_{2,h}$ evaluated at $(x_0,y_0)$. i.e.;
	\begin{equation}\label{Schro2Dsum}
	\begin{split}
   		&\mathcal{H}_{2,h}(V_{2}(x,y))\psi(x,y) \mid_{x=x_0,\,y=y_0} \\ = & \,\, \mathcal{A}_{x_0,h}(V_{2}(x_0,y) )\varphi_{x_0}(y)\mid_{y=y_0}+\, \mathcal{B}_{y_0,h}(V_{2}(x,y_0))\phi_{y_0}(x)\mid_{x=x_0}.
	\end{split}
	\end{equation}
	
	We also define the following spectral problems, 
	\begin{equation}\label{spec_proby_0}
			\mathcal{A}_{x_0,h}(V_2(x_0,y))\varphi_{x_0,n,h}(y) = \kappa_{x_0,n,h}\varphi_{x_0,n,h}(y),
		\end{equation}
		\begin{equation}\label{spec_probx_0}
			\mathcal{B}_{y_0,h}(V_2(x,y_0))\phi_{y_0,m,h}(x) = \rho_{y_0,m,h}\phi_{y_0,m,h}(x),
		\end{equation}
where $\kappa_{x_0,n,h}$ and $\varphi_{x_0,n,h}$ for $n = 1,\cdots,N^{\lambda}_h$ (resp. $\rho_{y_0,m,h}$ and $\phi_{y_0,m,h}$ for $m = 1,\cdots,M^{\lambda}_h$) are the decreasing negative eigenvalues and associated $L^2$-normalized eigenfunctions of the operator $(\ref{Schro2Dy0})$ (resp. $(\ref{Schro2Dx0})$) and $N^{\lambda}_h$ (resp. $M^{\lambda}_h$) is the number of negative eigenvalues less then $\lambda$.

Multiplying (\ref{spec_proby_0}) by $\phi_{y_0,m,h}(x)$ and (\ref{spec_probx_0})  by $\varphi_{x_0,n,h}(y)$ and adding the results gives
\begin{eqnarray}\label{spec_proby_0x_0}
				\left\{\mathcal{A}_{x_0,h}(V_2(x_0,y)) +\mathcal{B}_{y_0,h}(V_2(x,y_0))\right\} \varphi_{x_0,n,h}(y)\phi_{y_0,m,h}(x) \nonumber\\
			= (\kappa_{x_0,n,h}+ \rho_{y_0,m,h}) \varphi_{x_0,n,h}(y)\phi_{y_0,m,h}(x),
		\end{eqnarray}

	In particular for $x=x_0$ and $y=y_0$ and using (\ref{Eq_spec2D}) and (\ref{Schro2Dsum}) we have:
	\begin{eqnarray}\label{spec_res}
			 \varphi_{x_0,n,h}(y_0)\phi_{y_0,m,h}(x_0) &=& \psi_{k,h} (x_0,y_0), \\
			 \kappa_{x_0,n,h}+ \rho_{y_0,m,h} & = & \mu_{k,h},
		\end{eqnarray}
 with $k = 1,\cdots, K_h^{\lambda}$, $n=1,\cdots, N_h^{\lambda}$ and $m=1,\cdots,M_h^{\lambda}$ and  $K_h^{\lambda}=N_h^{\lambda}\times M_h^{\lambda}$.

		So formula (\ref{scsa2d}) can be written as follows, 
\begin{equation}
\scriptstyle V_2(x_0,y_0) = \scriptstyle -\lambda+ \scriptstyle \displaystyle
 \lim_{h\rightarrow0} \scriptstyle \left(\frac{h^{2}}{L^{cl}_{2,\gamma}} \displaystyle \sum_{n=1}^{N_{h}^{\lambda}} \displaystyle \sum_{m=1}^{M_{h}^{\lambda}} \scriptstyle
 \left(\lambda-({\kappa}_{x_0,n,h}+{\rho}_{y_0,m,h})\right)^{\gamma}{\varphi}^{2}_{x_0,n,h}(y_0){\phi}^{2}_{y_0,m,h}(x_0)\right) ^{\frac{1}{1+\gamma}},
\end{equation}

This formula shows that the 2D function $V_2$ can be estimated at $(x_0,y_0)$ from a tensor products resulting from solving spectral problems for 1D\\
 Schr\"odinger operators. 


\subsection{Application to images }
We denote  $I$ an image on space of square matrices $\mathcal{M}_{N \times N}(\mathbb{R}^+)$. The discretization of the eigenvalue problem ($\ref{Eq_spec2D}$) is given by the following eigenvalue problem,
\begin{equation}\label{Sch_sep}
		{H}_{2,h}(I)\, \underline{\psi}_{([i,j],k,h} = {\mu}_{[i,j],k,h}\,\underline{\psi}_{(i,j),k,h},
\end{equation}
where ${\mu}_{[i,j],k,h}$ and $\underline{\psi}_{[i,j],k,h}$, for $k = 1,\cdots,K_h^{\lambda}$ with $K_h^{\lambda}< N \times N$, refer to the negative eigenvalues with ${\mu}_{[i,j],1,h}<\cdots<{\mu}_{[i,j],K_h^{\lambda},h}<\lambda$ and associated $l^2$-normalized eigenvectors respectively of the 2D discretized semi-classical Schr\"odinger operator ${H}_{2,h}$ and $i,j = 1,\cdots,N$ refer to the $i^{th}$ row and $j^{th}$ column of the matrix respectively.

To solve the 2D eigenvalue problem $(\ref{Sch_sep})$ and as described in the previous subsection, the idea consists in solving 1D eigenvalues problems. This means for the image, solving the problem rows by rows and columns by columns which simplifies the computations in terms of complexity and computation time and especially allows for parallel computing.

In discrete case, the operators $(\ref{Schro2Dy0})$ and $(\ref{Schro2Dx0})$ are given respectively by:
	\begin{equation}\label{Schro2Di}
   		A_{i, h}(I[i,:])\underline{\varphi}_{i} = -h^{2}D_2\underline{\varphi}_{i} - \text{diag}\left(\frac{1}{2}I[i,:]\right)\underline{\varphi}_{i},
	\end{equation}
	\begin{equation}\label{Schro2Dj}
   		B_{j, h}(I[:,j]) \underline{\phi}_{j}= -h^{2}D_2 \underline{\phi}_{j} - \text{diag}\left(\frac{1}{2}I[:,j]\right)\underline{\phi}_{j},
	\end{equation}
where $D_2$ is a second order differentiation matrix obtained using the Fourier pseudo-spectral method $\cite{Boyd2000, Trefethen2000}$, $\text{diag}\left( \frac{1}{2}I[i,:] \right)$ and $\text{diag}\left( \frac{1}{2}I[:,j]\right)$ are the diagonal matrix of the 1D signal for the $i^{th}$ row and $j^{th} column$ respectively. 

Then the associated spectral problems are given by,
	\begin{equation}\label{spec_probi}
		A_{i, h}(I[i,:]) \underline{\varphi}_{i,n,h} = \kappa_{i,n,h}\underline{\varphi}_{i,n,h},
	\end{equation}
	\begin{equation}\label{spec_probj}
		B_{j, h}(I[:,j]) \underline{\phi}_{j,m,h} = \rho_{j,m,h}\underline{\phi}_{j,m,h},
	\end{equation}

In particular, for the pixel $[i,j]$, we solve the eigenvalue problem $(\ref{spec_probi})$ (resp. $(\ref{spec_probj})$), and then we take all the negative eigenvalues $\kappa_{i,n,h}$ (resp. $\rho_{j,m,h}$) and the $j^{th}$ (resp. $i^{th}$) associated $l^2$-normalized eigenvectors $\underline{\varphi}_{i,n,h}$ for $n = 1,\cdots,N^h_{\lambda}$ (resp. $\underline{\phi}_{j,m,h}$ for $m = 1,\cdots,M^h_{\lambda}$). Hence, we obtain, 
	\begin{eqnarray}\label{spec_res_ij}
			\underline{\varphi}_{i,n,h}[j]\underline{\phi}_{j,m,h}[i] &=& \underline{\psi}_{k,h}[i,j], \\
			 \kappa_{i,n,h}+ \rho_{j,m,h} & = & \mu_{k,h},
		\end{eqnarray}

Then, based on the Theorem \ref{Main_th}, the reconstruction of the image is done pixel by pixel as it is often the case in image processing as follows:

\begin{proposition} \label{SCS2D_Def}
Let $I \in \mathcal{M}_{N\times N} (\mathbb{R}_+)$ be a positive real valued square matrix. Then, the representation of $I$ using the SCSA method is given by the following formula: $\forall\,(i,j)\in \{1,2,\cdots,N\}^2$,
\begin{equation}\label{scsa2ddis}
\scriptstyle I_{h,\gamma,\lambda}[i,j]= -\lambda+ \left(\frac{h^{2}}{L^{cl}_{2,\gamma}}  \displaystyle \sum_{n=1}^{N_{h}^{\lambda}}\displaystyle  \sum_{m=1}^{M_{h}^{\lambda}} \left(\lambda-({\kappa}_{i,n,h}+{\rho}_{j,m,h})\right)^{\gamma}\underline{\varphi}^{2}_{i,n,h}[j]\underline{\phi}^{2}_{j,m,h}[i]\right) ^{\frac{1}{1+\gamma}},
\end{equation}
where $h\in \mathbb{R}^{*}_{+}$, $\gamma\in \mathbb{R}_{+}$, $ \lambda \in \mathbb{R}_{-}$, and $L^{cl}_{2,\gamma}$, known as the suitable universal semi-classical constant, is given by $ \left(\ref{Lcl} \right)$.
Moreover, ${\kappa}_{i,n,h}$ $\left( \text{resp. } {\rho}_{j,m,h} \right)$ are the negative eigenvalues of the one dimensional semi-classical Schr\"odinger operator given by $(\ref{Schro2Di})$, $\left( \text{resp. } (\ref{Schro2Dj})\right)$ with ${\kappa}_{i,1,h} < \cdots <{\kappa}_{i,N_{h}^{\lambda},h}< \lambda$ \\
$\left( \text{resp. } {\rho}_{j,1,h} <\cdots <{\rho}_{j,M_{h}^{\lambda},1}< \lambda \right )$, $N_{h}^{\lambda}$ $ \left(\text{resp. } M_{h}^{\lambda}\right)$ is the number of the negative eigenvalues smaller than $\lambda$, and $\underline{\varphi}_{i,n,h}$ $\left( \text{resp. } \underline{\phi}_{j,m,h} \right)$ are the associated $l^2$-normalized eigenvectors.
\end{proposition}

\subsection{Algorithm description}

The reconstruction of the pixel $[i,j]$ requires solving one dimensional eigenvalue problems corresponding to the row $i$ and the column $j$ respectively. The element $[i,j]$ is then considered twice, which justifies the $\frac{1}{2}$ in the potential's one dimensional operators.

The approach is summarized in the following steps, \\

\quad \textbf{Step 1}: solve the eigenvalues problem (\ref{spec_probi}) with a potential $\frac{1}{2} I[i,:]$, and then take all the negative eigenvalues $\kappa_{i,n,h}$ and the $j^{th}$ associated $l^2$-normalized eigenvectors of $\underline{\varphi}_{i,n,h}$ for $n=1,\cdots,N_h^{\lambda}$.\\

 \quad \textbf{Step 2}: solve the eigenvalues problem (\ref{spec_probj}) with a potential $\frac{1}{2}I[:,j]$, and then take all the negative eigenvalues ${\rho}_{j,m,h}$ and the $i^{th}$ associated $l^2$-normalized eigenvectors of $\underline{\phi}_{j,m,h}$ for $m = 1,\cdots, M_h^{\lambda}$.\\

\quad \textbf{Step 3}: reconstruct the image using formula (\ref{scsa2ddis}). \\

The figure below illustrates the principle of the proposed algorithm.

\newpage
\setlength{\unitlength}{0.8cm}
\begin{picture}(7,3)(-7,-2)
\put(-7,2){The image}
\multiput(-2.5,1)(0.4,0){21} {\line(1,0){0.2}}%
\multiput(-2.5,0.8)(1.4,0){7} {\line(0,1){0.2}}
\multiput(-2.5,0.4)(1.4,0){7} {\line(0,1){0.2}}
\multiput(-2.5,0.0)(1.4,0){7} {\line(0,1){0.2}}
\multiput(-2.5,-0.4)(1.4,0){7} {\line(0,1){0.2}}
\multiput(-2.5,-0.4)(0.4,0){21} {\line(1,0){0.2}}%
\multiput(-2.5,-0.8)(1.4,0){7} {\line(0,1){0.2}}
\multiput(-2.5,-1.2)(1.4,0){7} {\line(0,1){0.2}}
\multiput(-2.5,-1.6)(1.4,0){7} {\line(0,1){0.2}}
\multiput(-2.5,-2.0)(1.4,0){7} {\line(0,1){0.2}}
\multiput(-2.5,-1.8)(0.4,0){21} {\line(1,0){0.2}}%
\multiput(-2.5,-2.4)(1.4,0){7} {\line(0,1){0.2}}
\multiput(-2.5,-2.8)(1.4,0){7} {\line(0,1){0.2}}
\multiput(-2.5,-3.2)(1.4,0){7} {\line(0,1){0.2}}
\multiput(-2.5,-3.2)(0.4,0){21} {\line(1,0){0.2}}%
\multiput(-2.5,-3.6)(1.4,0){7} {\line(0,1){0.2}}
\multiput(-2.5,-4.0)(1.4,0){7} {\line(0,1){0.2}}
\multiput(-2.5,-4.4)(1.4,0){7} {\line(0,1){0.2}}
\multiput(-2.5,-4.6)(0.4,0){21} {\line(1,0){0.2}}%
\multiput(-2.5,-4.8)(1.4,0){7} {\line(0,1){0.2}}
\multiput(-2.5,-5.2)(1.4,0){7} {\line(0,1){0.2}}
\multiput(-2.5,-5.6)(1.4,0){7} {\line(0,1){0.2}}
\multiput(-2.5,-6)(1.4,0){7} {\line(0,1){0.2}}
\multiput(-2.5,-6)(0.4,0){21} {\line(1,0){0.2}}%
\multiput(-2.5,-6.4)(1.4,0){7} {\line(0,1){0.2}}
\multiput(-2.5,-6.8)(1.4,0){7} {\line(0,1){0.2}}
\multiput(-2.5,-7.2)(1.4,0){7} {\line(0,1){0.2}}
\multiput(-2.5,-7.4)(0.4,0){21} {\line(1,0){0.2}}%
\put(-2.35,0.25){\small{I[1,1]}}
\put(0.5,0.25){\small {I[1,j]}}
\put(4.7,0.25){\small{I[1,N]}}
\put(-2.3,-4){\small{I[i,1]}}
\put(0.5,-4){$\mathbf{I[i,j]}$}
\put(4.7,-4){\small{I[i,N]}}
\put(-2.3,-6.8){\small{I[N,1]}}
\put(0.5,-6.8){\small{I[N,j]}}
\put(4.6,-6.8){\small{I[N,N]}}
\put(0.1,1.15){\line(1,0){1.8}}
\put(0.1,-7.6){\line(0,1){8.75}}
\put(0.1,-7.6){\line(1,0){1.8}}
\put(1.9,-7.6){\line(0,1){8.75}}
\put(-2.65,-3){\line(1,0){8.75}}
\put(-2.65,-4.8){\line(0,1){1.8}}
\put(-2.65,-4.8){\line(1,0){8.75}}
\put(6.1,-4.8){\line(0,1){1.8}}
\put(-4.4,-4){$i^{th}$ row}
\put(0,-8.3){$j^{th}$ column}
%
%
\put(-7,-9){For the $i^{th}$row}
\put(-6,-12){$\left( \begin{array}{c} \frac{1}{2}{[i,1]} \\
 \cdot \\
 \cdot \\
 \frac{1}{2}I[i,j] \\
 \cdot \\
 \cdot \\
 \frac{1}{2}I[i,N] \\ \end{array} \right)$}
\put(-4.5,-11.85){\oval(1.9,0.8)}
\put(-2.55,-10){$\left(\begin{array}{cccccccccc}  \kappa_{i,1,h} & & \quad \,\, \kappa_{i,2,h} & & \cdot & \,\,\, \kappa_{i,N_h^{\lambda},h} & & 0 & \cdot & 0 \,\,\, \end{array} \right)$}
\put(2.7,-9.9){\oval(10,0.7)}
\put(-2.65,-13){$\left( \begin{array}{ccccccc} \varphi_{i,1,h}[1] & \varphi_{i,2,h}[1] & \cdot &\varphi_{i,N_h^{\lambda},h}[1] &0 &\cdot & 0 \\
 \cdot & \cdot & \cdot & \cdot & \cdot & \cdot & \cdot \\
 \cdot & \cdot & \cdot & \cdot & \cdot & \cdot & \cdot \\
 \varphi_{i,1,h}[j] & \varphi_{i,2,h}[j] & \cdot & \varphi_{i,N_h^{\lambda},h}[j] & 0 &\cdot & 0\\
 \cdot & \cdot & \cdot & \cdot & \cdot & \cdot & \cdot \\
 \cdot & \cdot & \cdot & \cdot & \cdot & \cdot & \cdot \\
 \varphi_{i,1,h}[N] & \varphi_{i,2,h}[N] & \cdot &\varphi_{i,N_h^{\lambda},h}[N] & 0 & \cdot & 0 \\\end{array} \right)$}
\put(2.7,-12.85){\oval(10,0.7)}
%
%
\put(-7,-16.5){For the $j^{th}$colunm}
\put(-6,-19.5){$\left( \begin{array}{c} \frac{1}{2}I[1,j] \\
\cdot \\
\cdot \\
\frac{1}{2}I[i,j] \\
\cdot \\
\cdot \\
\frac{1}{2}I[N,j] \\ \end{array} \right)$}
\put(-4.5,-19.4){\oval(1.9,0.8)}
\put(-2.6,-17.5){$\left(\begin{array}{cccccccccc} \quad \rho_{j,1,h} & \, & \rho_{j,2,h} & & \, \cdot & \,\,\,\rho_{j,M_h^{\lambda},h} & & 0& \,\, \cdot &  0 \,\,\,\end{array} \right)$}
\put(2.7,-17.37){\oval(10,0.7)}
\put(-2.65,-20.5){$\left( \begin{array}{ccccccc} \phi_{i,1,h}[1] & \phi_{i,2,h}[1] & \cdot & \phi_{i,M_h^{\lambda},h}[1] &\ 0 & \cdot & 0 \\
 \cdot &\cdot &\cdot &\cdot & \cdot &\cdot & \cdot \\
 \cdot &\cdot & \cdot & \cdot & \cdot &\cdot & \cdot \\
 \phi_{i,1,h}[i] & \phi_{j,2,h}[i] & \cdot & \phi_{j,M_h^{\lambda},h}[i] & 0 & \cdot & 0 \\
 \cdot &\cdot & \cdot &\cdot & \cdot &\cdot & \cdot \\
 \cdot & \cdot &\cdot &\cdot & \cdot &\cdot & \cdot \\
 \phi_{i,1,h}[N] & \phi_{j,2,h}[N] &\cdot & \phi_{i,M_h^{\lambda},h}[N] & 0 &\cdot & 0 \\ \end{array} \right)$}
\put(2.7,-20.37){\oval(10,0.7)}
\put(-4,-26){Fig. 0. Principle behind the 2D SCSA algorithm.}
\end{picture}

\newpage


Based on the above discussions, the proposed algorithm may be stated as follows.\\
\vspace{2pt}\hrule\vspace{4pt}
\noindent \textbf{  Algorithm 1:} The 2D SCSA algorithm\par\nobreak
\vspace{2pt}\hrule\vspace{4pt}

$\,\,\,$ \small{\textbf{Input:} The image to be analyzed}\\

$\,\,\,$ \small{\textbf{Output:} Estimated image}\\

$\,\,\,$ Following are the steps of the algorithm:

$\quad$ \small{\textbf{Step 1:} Initialize  $h$, $\lambda$ and $\gamma$.}\\

$\quad$ \small{\textbf{Step 2:} Discretize the Laplace operator $D_2$.}\\

$\quad$ \small{\textbf{Step 3:} Solve 1D eigenvalue problems $(\ref{Schro2Di})$ and $(\ref{Schro2Dj})$ (for all rows $i$ and columns $j$ with $i,j = 1,\cdots, N$ respectively).}\\

$\quad$ \small{\textbf{Step 4: }Reconstruct the image using formula $(\ref{scsa2ddis})$}\\

\vspace{2pt}\hrule\vspace{4pt}

\section{Numerical results}\label{sec5}
Formula (\ref{scsa2ddis}), depends on three parameters: $\lambda$, $\gamma$ and $h$.
$\lambda$ gives information on the part of the signal to reconstruct \cite{Helffer2011}.
For sake of simplicity, we propose to take $\lambda =0$ in the following. 
Only the semi-classical parameter $h$ affects the computed eigenvalues and eigenfunctions since the operator depends on its values.
Also, it is well-known that the number of negative eigenvalues depends on $h$ such that as $h$ decreases $N_h$ and $M_h$ increases \cite{Helffer2011}. In practice, like the Fourier method, and for practical reasons, there is a trade-off between the number of elementary functions and the desired reconstruction accuracy. From the implementation point of view, it is better to have a good representation of the image with a small enough number of eigenvalues. So we will choose $h$ large enough to have a good reconstruction with a small number of eigenvalues.
Moreover, it has been shown that in $1$D SCSA method, the parameter $\gamma$ may improve the approximation of the signal for a given small number of negative eigenvalues \cite{Helffer2011}.
This means that for a given $h$ (i.e $N_h$, $M_h$), the estimation of the signal can be improved by changing the value of $\gamma$.

The experiments have been carried out on academic functions of two variables and standard testing images for most state-of-art algorithms and the effect of the parameters $\gamma$ and $h$ has been studied numerically. In the following some of this experiments are presented.\\

\noindent\textbf{Example 1.}~\newline
In this example, we consider the following function:
\begin{equation}\label{fct1}
  V_{2}(x,y)=\sin(\frac{1}{2}x^{2}+\frac{1}{4}y^{2}+3)\,\cos(2x+1-e^{y})+1.
\end{equation}
for $(x,y)\in[-1,3]\times[-1,3]$. In discrete case $V_2$ is given by $I$ where $x_i=n T_s$ and $y_j=m T_s$ for $n,m=-50,\cdots,150$ with $T_s=0.02$ and $i,j = 1,\cdots,N$.

Before estimating $I$, we study the influence of the design parameters $h$ and $\gamma$. By taking different values of $h$ and $\gamma$, and by estimating the variation of the mean square errors between $I$ and the estimation $I_{h,\gamma,0}$,
\begin{equation}\label{mse}
MSE = \frac{ \displaystyle \sum_{i=1}^{N} \sum_{j=1}^{N} \left(I[i,j] - I_{h,\gamma,\lambda}[i,j] \right)^2}{N \times N},
\end{equation}
where $N$ is the number of discrete points, we found the existence of a minimum at $h=6\times10^{-3}$ and $\gamma=4$ as illustrated in Figure \ref{ContourFun}. Then, we estimate $I$ using $I_{h,\gamma,0}$ with these optimal parameter values (see Figure \ref{ImageFun}). In particular, we show in Figure \ref{SignalFun} the original signal $I[20,:]$ and the estimated one $I_{0.006,4,0}[20,:]$. Morever, we have shown in Figure \ref{ErrorFun} the relative error between the function and its estimation.

\begin{figure}[h!]
 \centering
 \subfigure[]
 {\includegraphics[height=5cm,width=6cm]{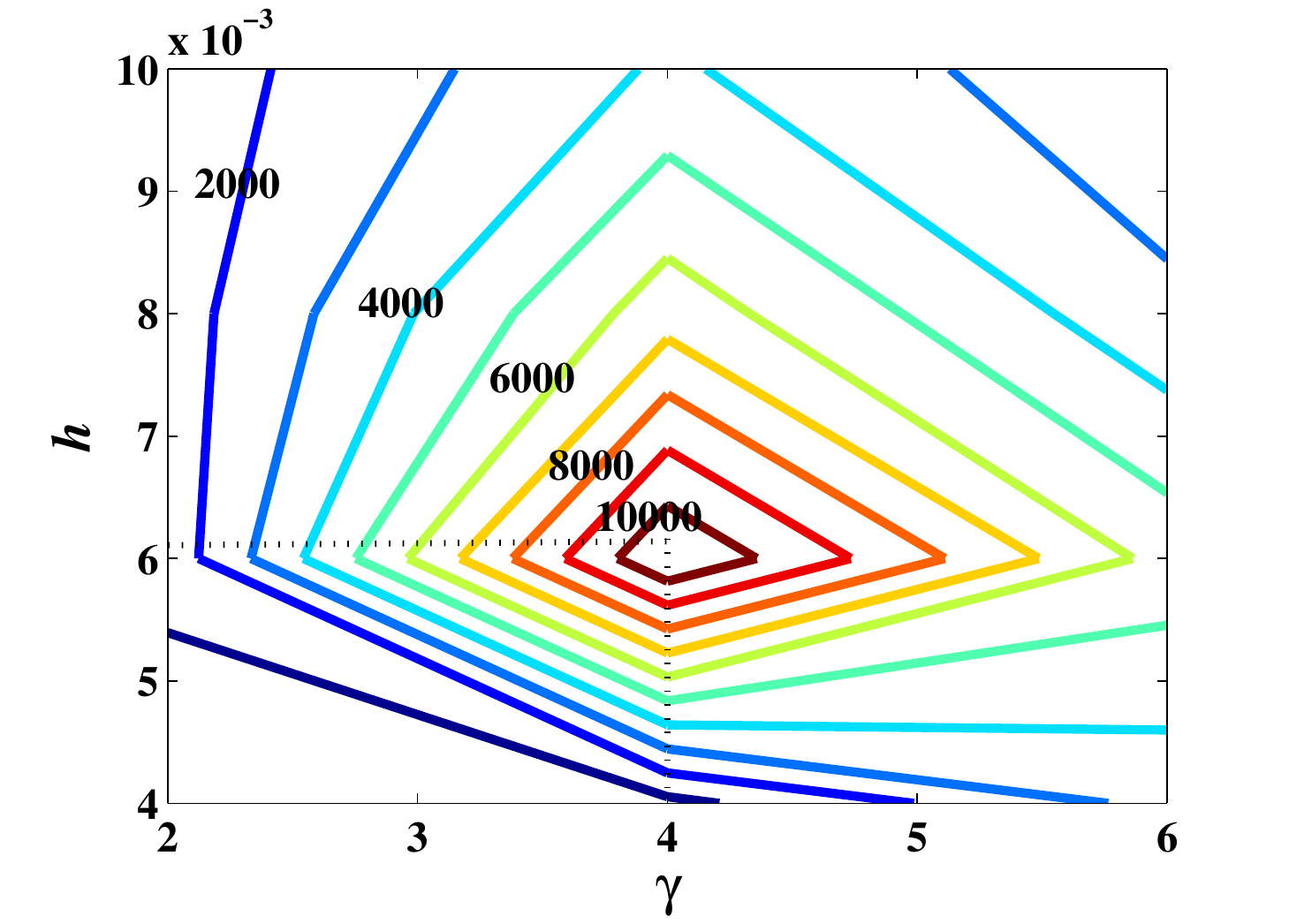}\label{ContourFun}}
 \subfigure[]
 {\includegraphics[height=5cm,width=6cm]{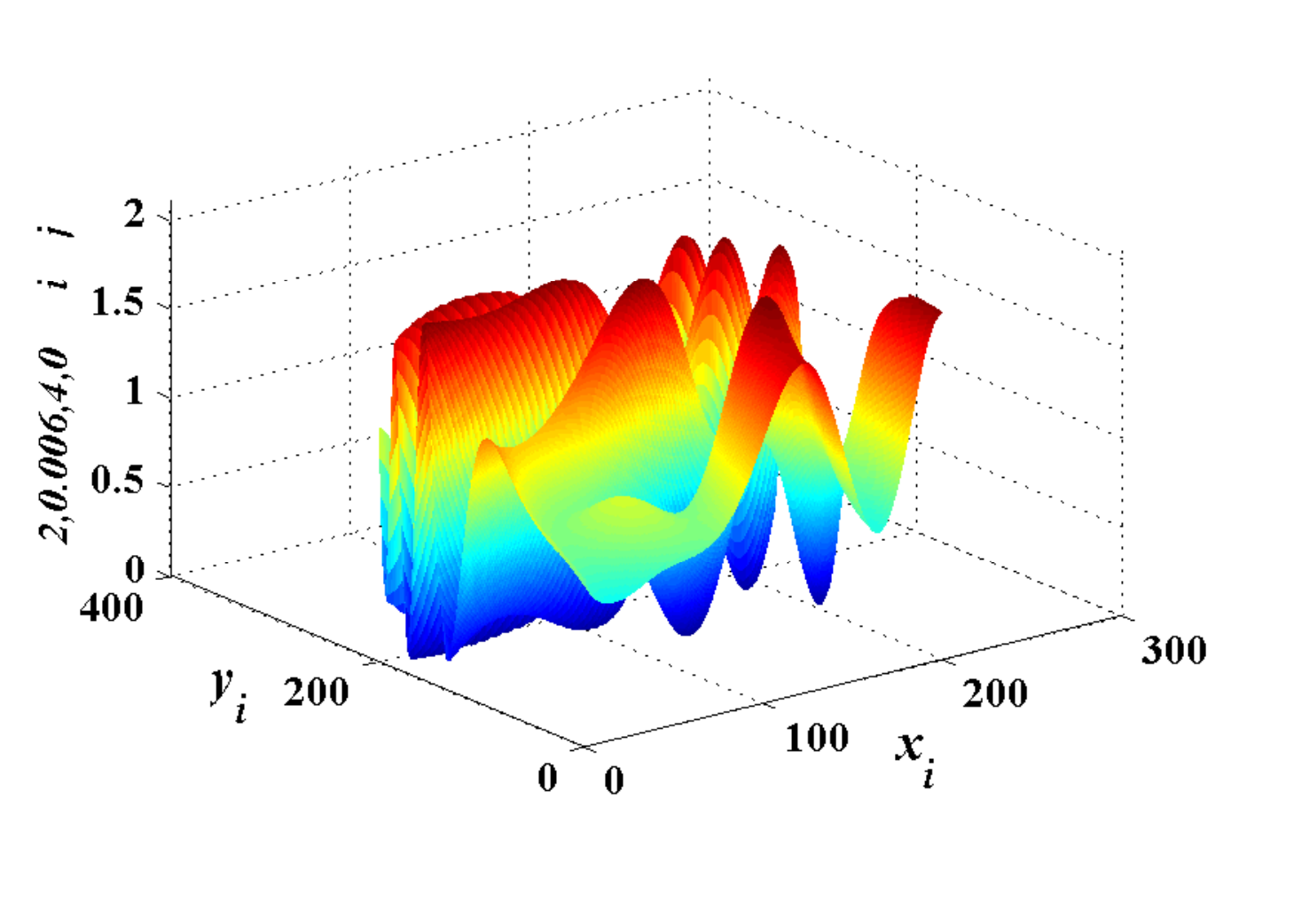}\label{ImageFun}}
 \subfigure[]
 {\includegraphics[height=5cm,width=6cm]{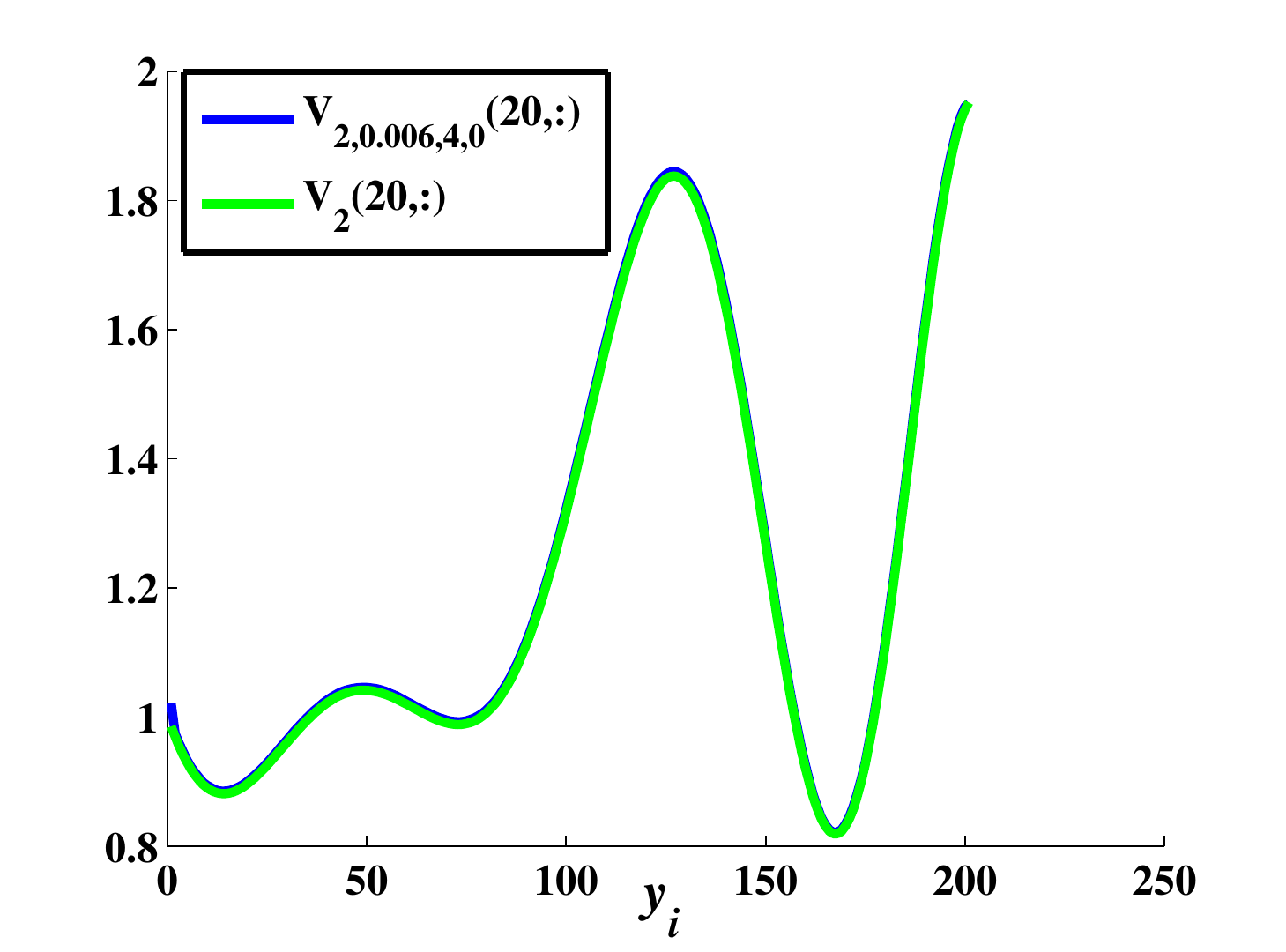}\label{SignalFun}}
 \subfigure[]
 {\includegraphics[height=5cm,width=6cm]{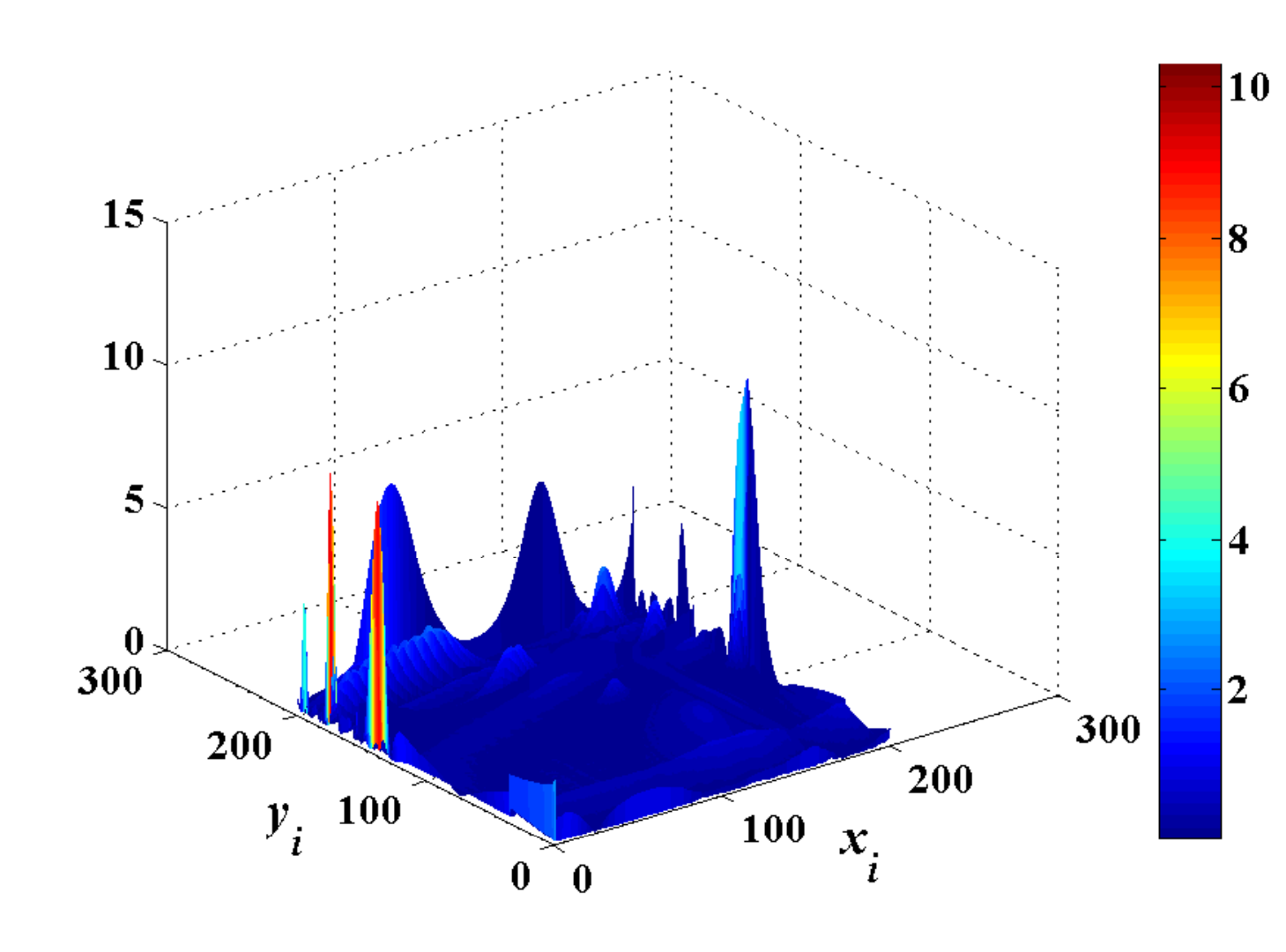}\label{ErrorFun}}

\caption{Example 1: $V_{2}(x,y)=\sin(\frac{1}{2}x^{2}+\frac{1}{4}y^{2}+3)\,\cos(2x+1-e^{y})+1$. $(a)$ The variation of $\frac{1}{\|I-I_{h,\gamma,0}\|^2}$. $(b)$ $I_{0.006,4,0}[i,j]$. $(c)$ $I[20,:]$ and $I_{0.006,4,0}[20,:]$ with $j=1,\cdots,N$. $(d)$ The relative error between the real function and its estimation.}
\end{figure}

\begin{remark}
We tested the algorithm for several examples. The obtained optimal value for $\gamma$ is $\gamma=4$.
\end{remark}

\noindent\textbf{Example 2.}~\newline
In this example, we consider a $440 \times 440$ pixels image, see Figure \ref{OriginalK}.
One can note the good reconstruction of this image in Figure \ref{reconsK}, for $h=0.21$ and $\gamma=4$. The relative error is shown in Figure \ref{ErrorK}.

\begin{figure}[h!]
 \centering
 \subfigure[]
 {\includegraphics[height=4cm,width=4cm]{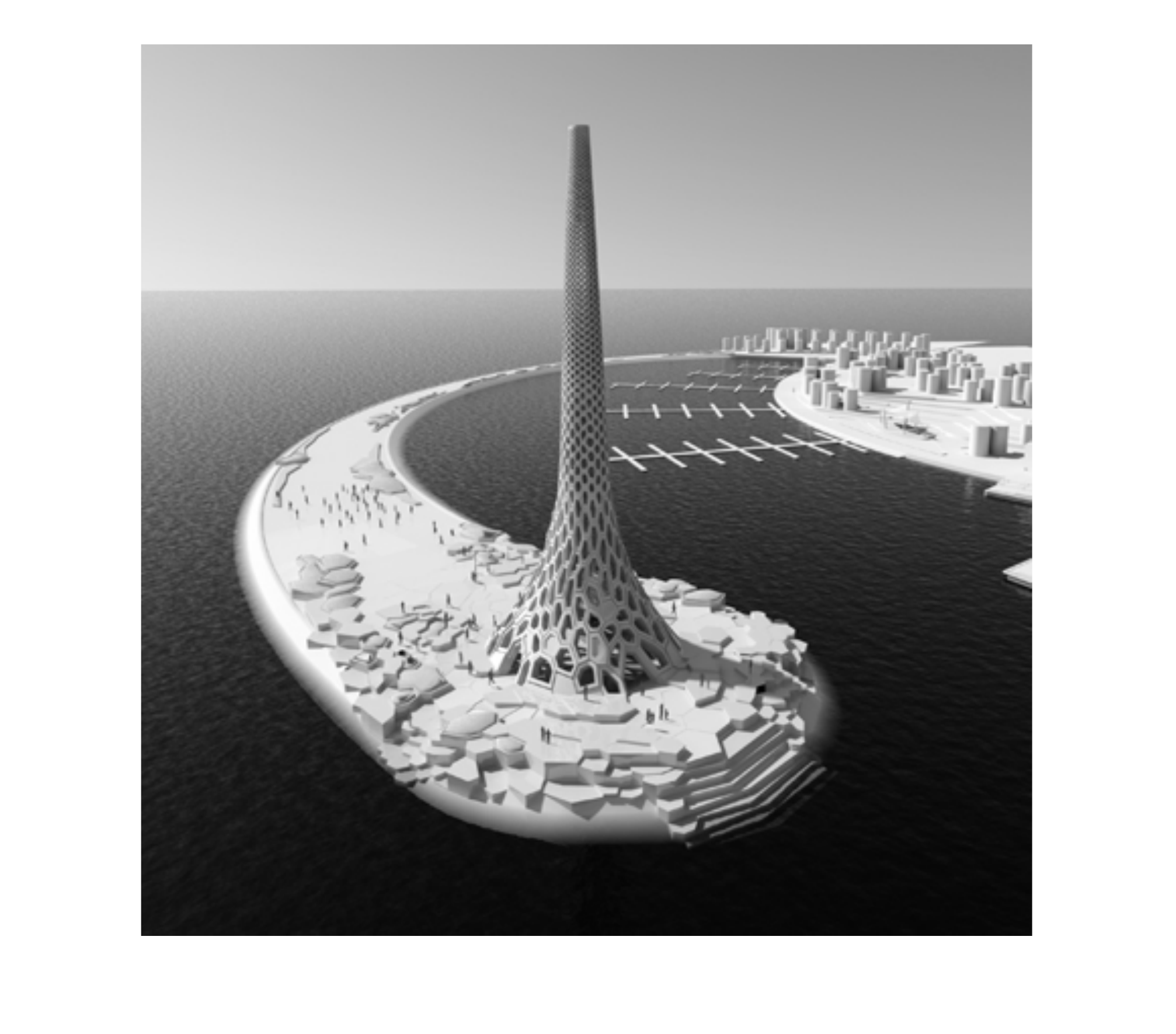}\label{OriginalK}}
 \subfigure[]
 {\includegraphics[height=4cm,width=4cm]{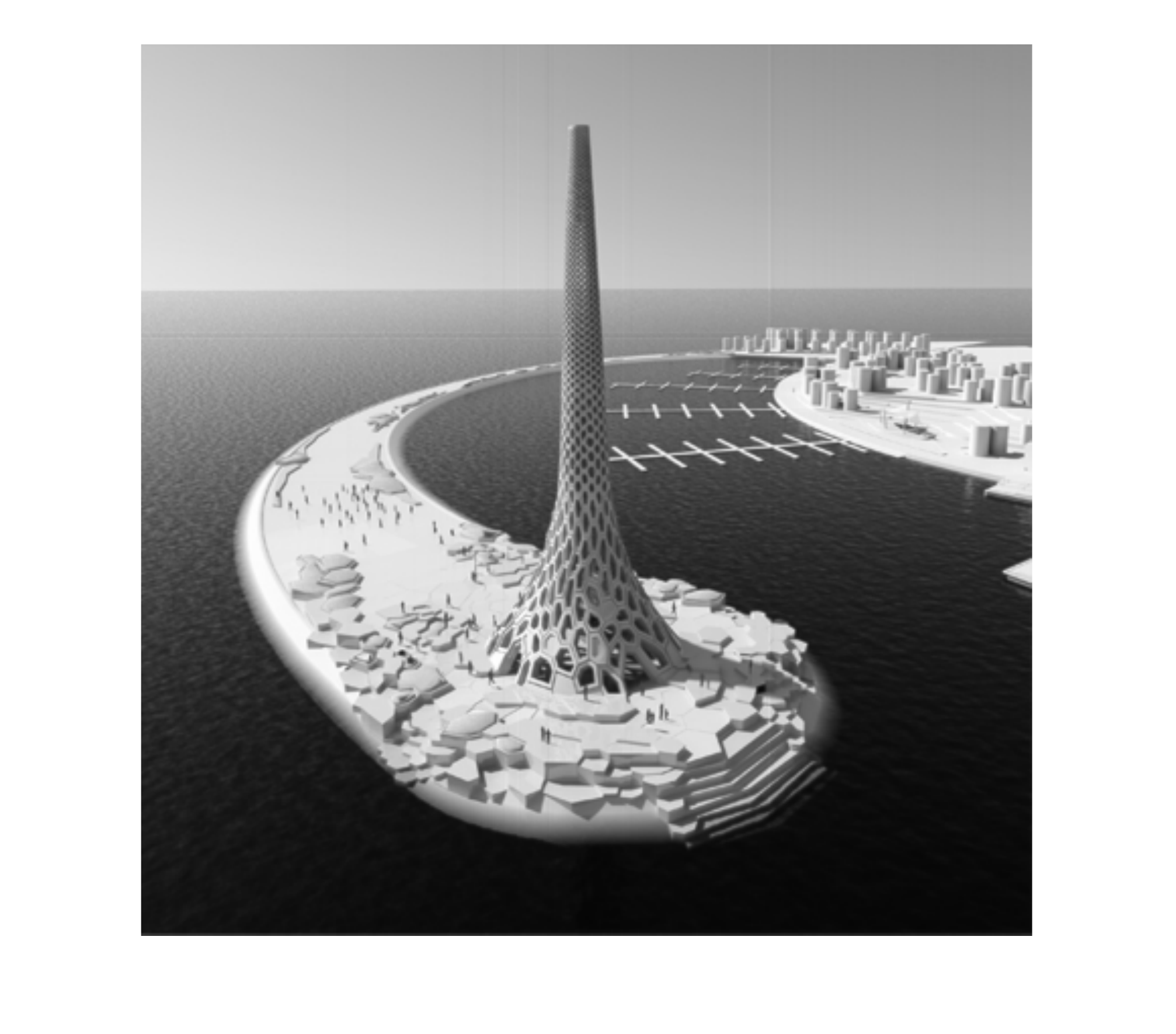}\label{reconsK}}
 \subfigure[]
 {\includegraphics[height=4cm,width=4cm]{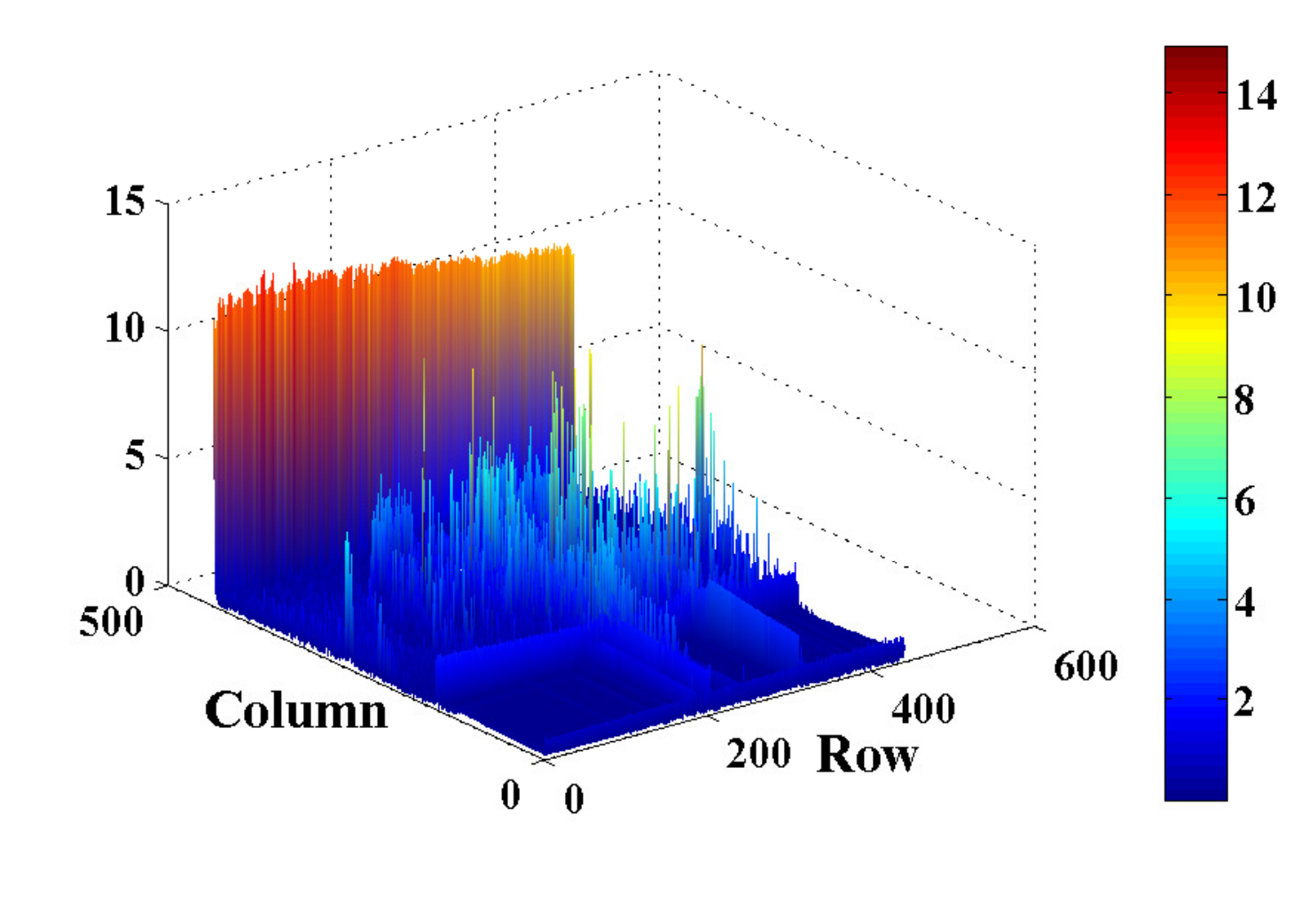}\label{ErrorK}}

 \caption{Example 2: $(a)$ Original image. $(b)$ Reconstructed image. $(c)$ The relative error between the original and reconstructed images.}
\end{figure}

\noindent\textbf{Example 3.}~\newline
In this example, we consider a $512 \times 512$ pixels Lena image \footnote{http://www.ece.rice.edu/~wakin/images/ }, see Figure \ref{OriginalL}.
Figure \ref{ContourL} illustrates the variation of the mean square error $(\ref{mse})$ for different values of $h$ and $\gamma$. The optimal values of $h$ and $\gamma$ are read  $0.2$ and $4$ respectively. Then, the image has been reconstructed using formula $(\ref{scsa2ddis})$ as illustrated in figures \ref{ReconsL} and \ref{ErrorL} respectively. 

	\begin{figure}[h!]
	 \centering
	 \subfigure[]
	 {\includegraphics[height=5cm,width=5cm]{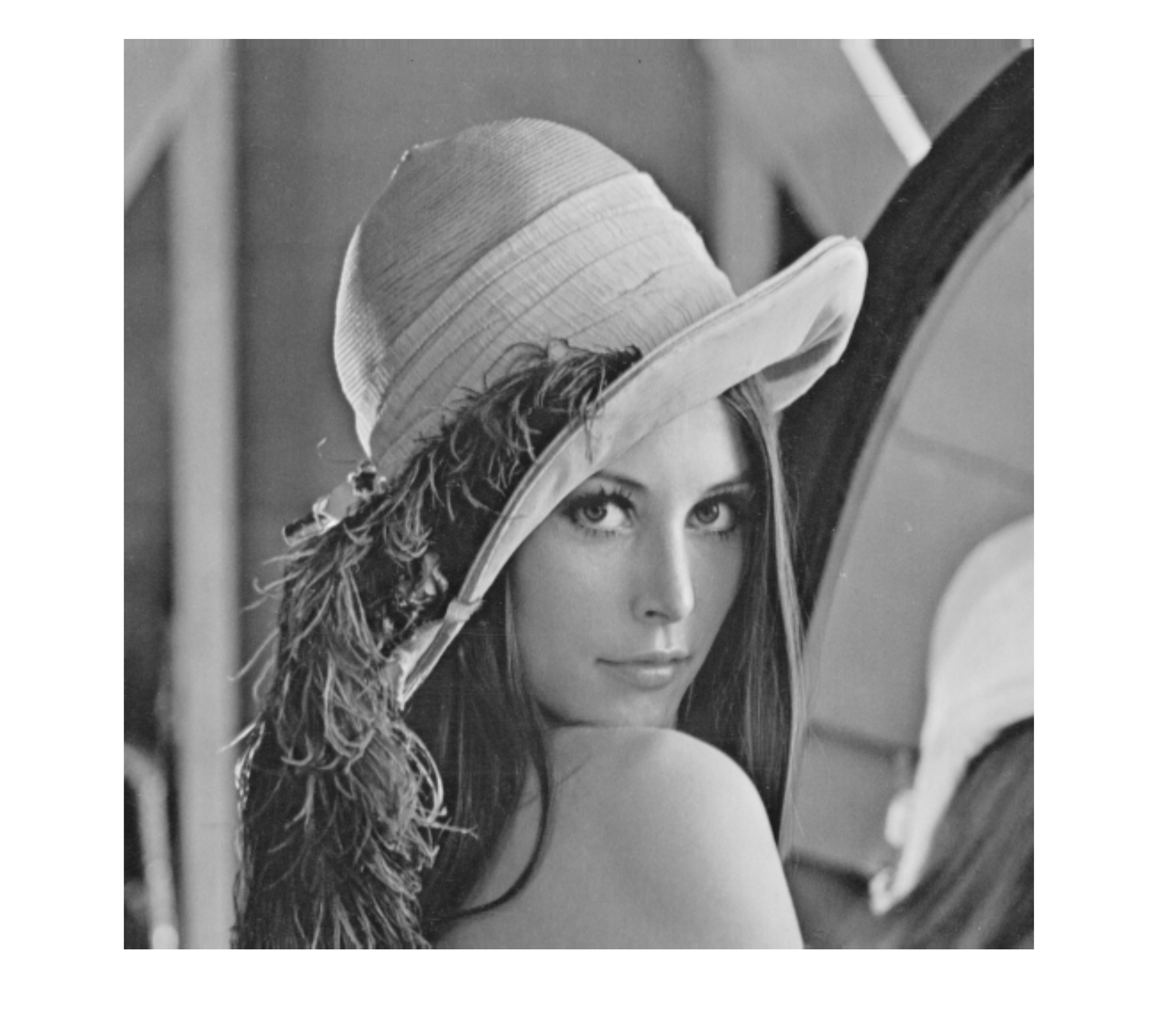}\label{OriginalL}}
	 	 \subfigure[]
	 {\includegraphics[height=5cm,width=5cm]{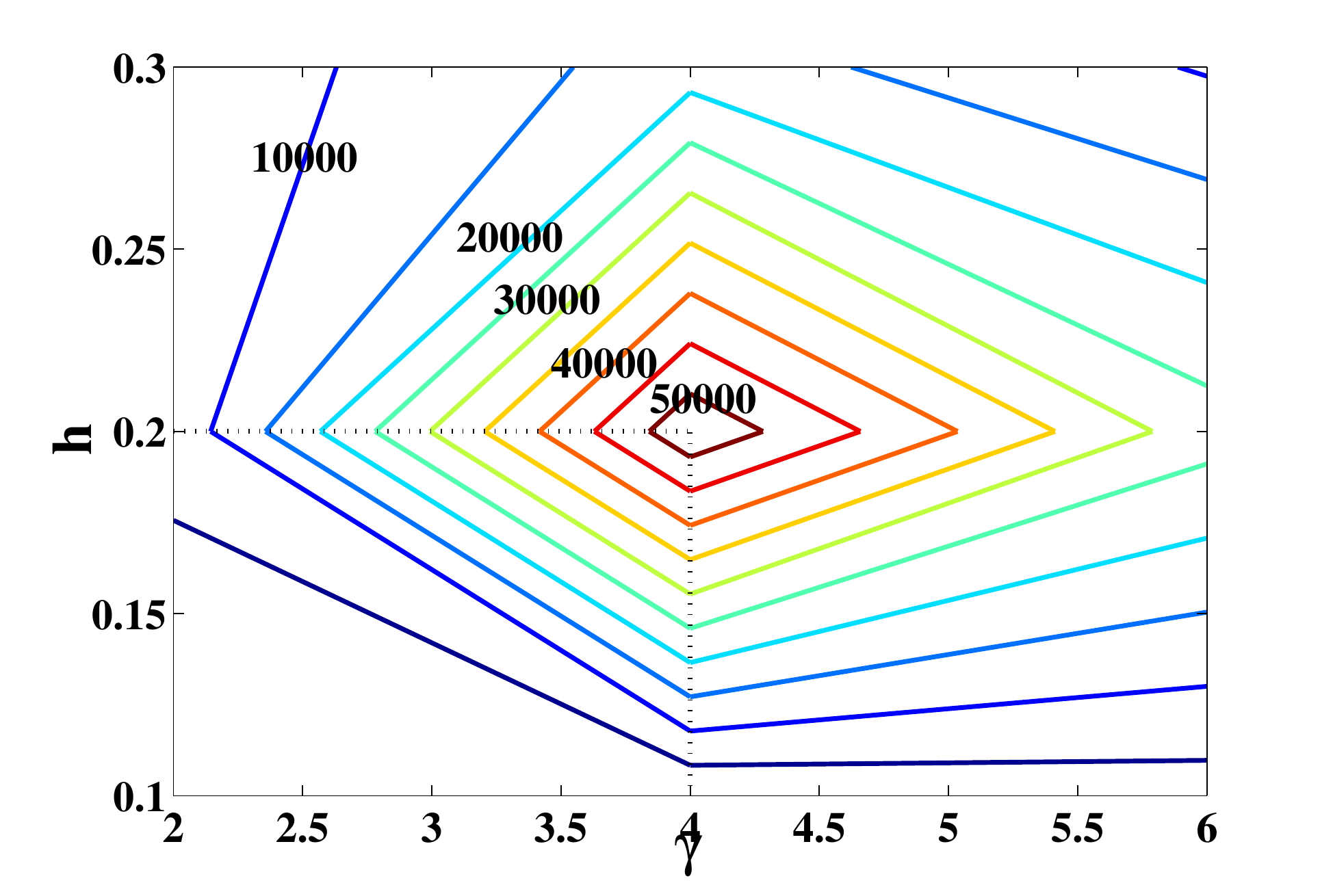}\label{ContourL}}
	 \subfigure[]
	 {\includegraphics[height=5cm,width=5cm]{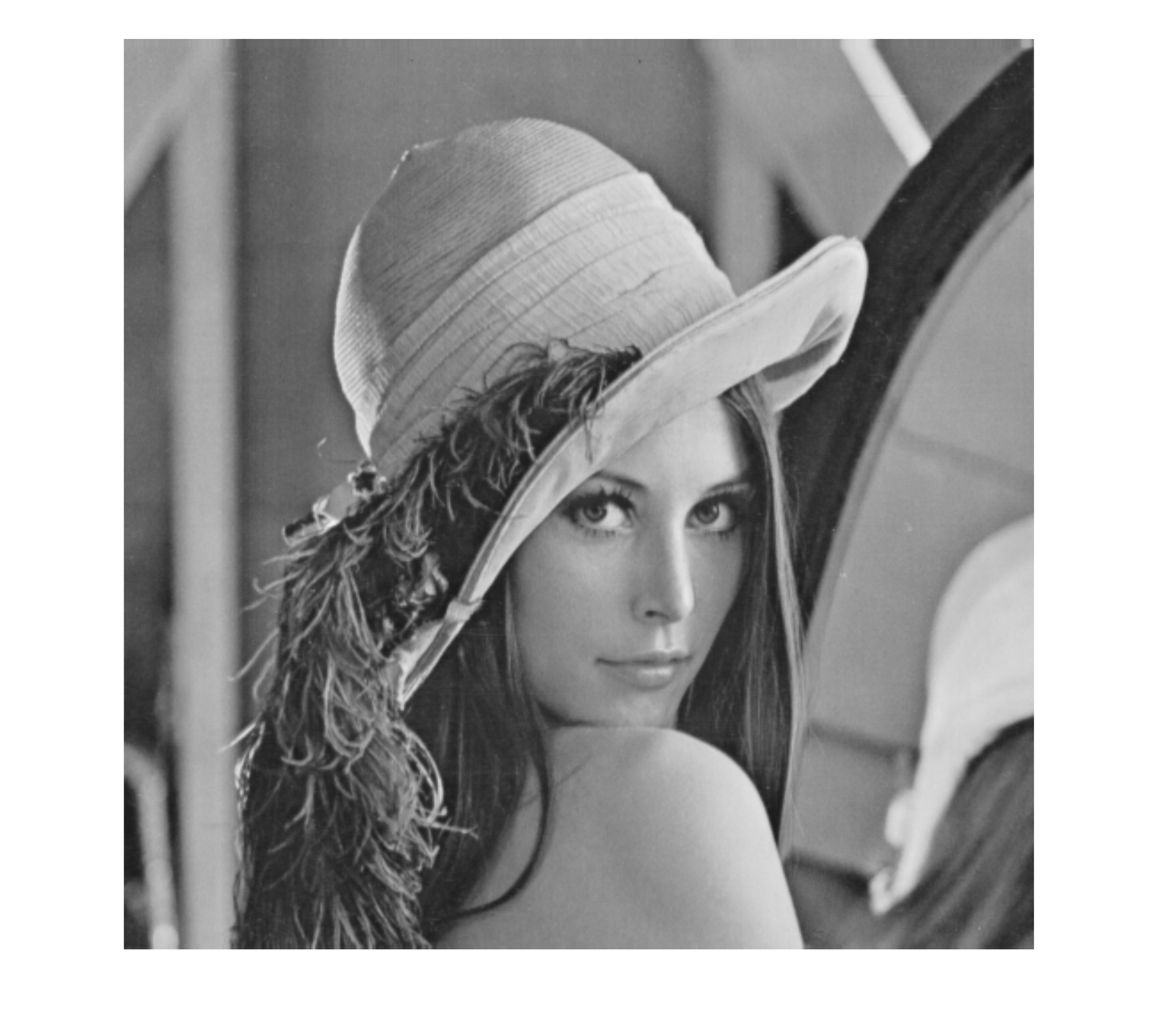}\label{ReconsL}}
	 \subfigure[]
	 {\includegraphics[height=5cm,width=5cm]{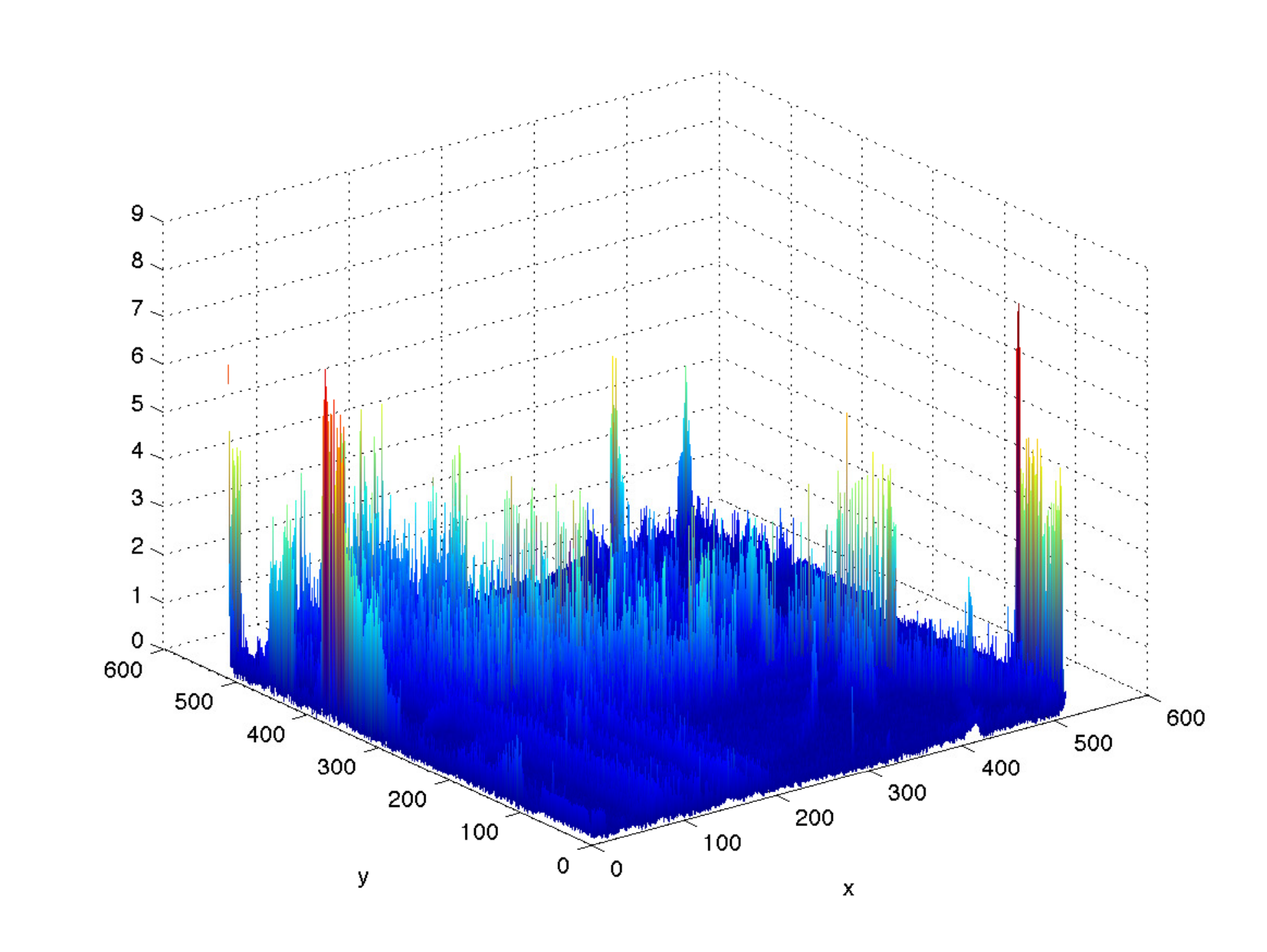}\label{ErrorL}}

	 \caption{Example 3: Lena's image. $(a)$ Original image. $(b)$ Reconstructed image. $(c)$ The relative error between the original and reconstructed images.}
	\end{figure}

Figure \ref{VPN2} shows the behavior of the number of negative eigenvalues for all rows and columns. It is clear that this number decreases when $h$ increases. 

\begin{figure}[h]
 \centering
  {\includegraphics[width=5cm]{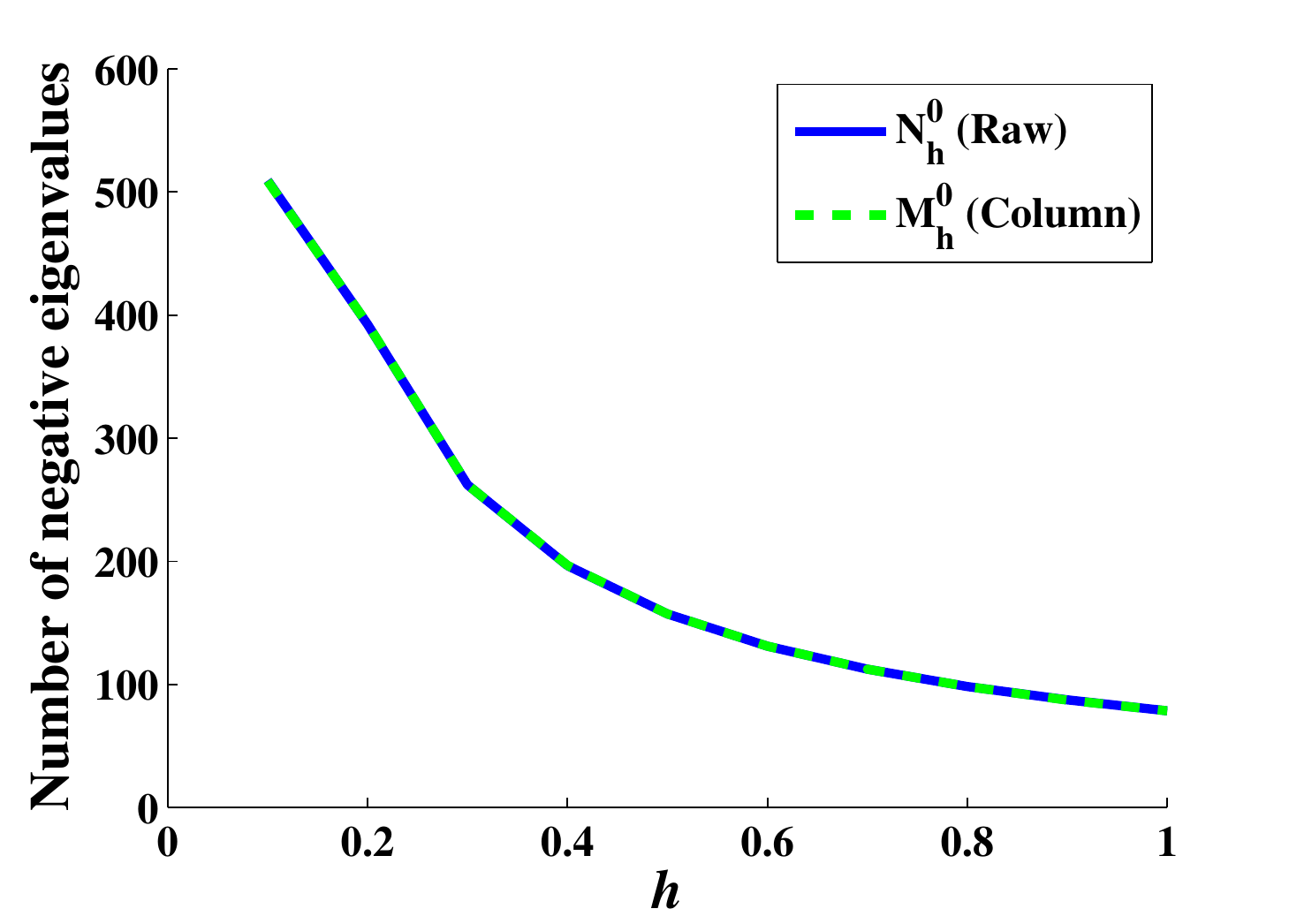}}
\caption{Number of negative eigenvalues.} \label{VPN2}
\end{figure}

The figures \ref{1egn} and \ref{1egnz} illustrate the localization property of the first $L^2$-normalized eigenfunctions which corresponds to the largest peak in the image.
The last $L^2$-normalized eigenfunction is given in figures \ref{512egn} and \ref{512egnz}. It contains several peaks, they represent the details in the image.

\begin{figure}[!h] 
 \centering
 \subfigure[]
 {\includegraphics[height=5cm,width=5cm]{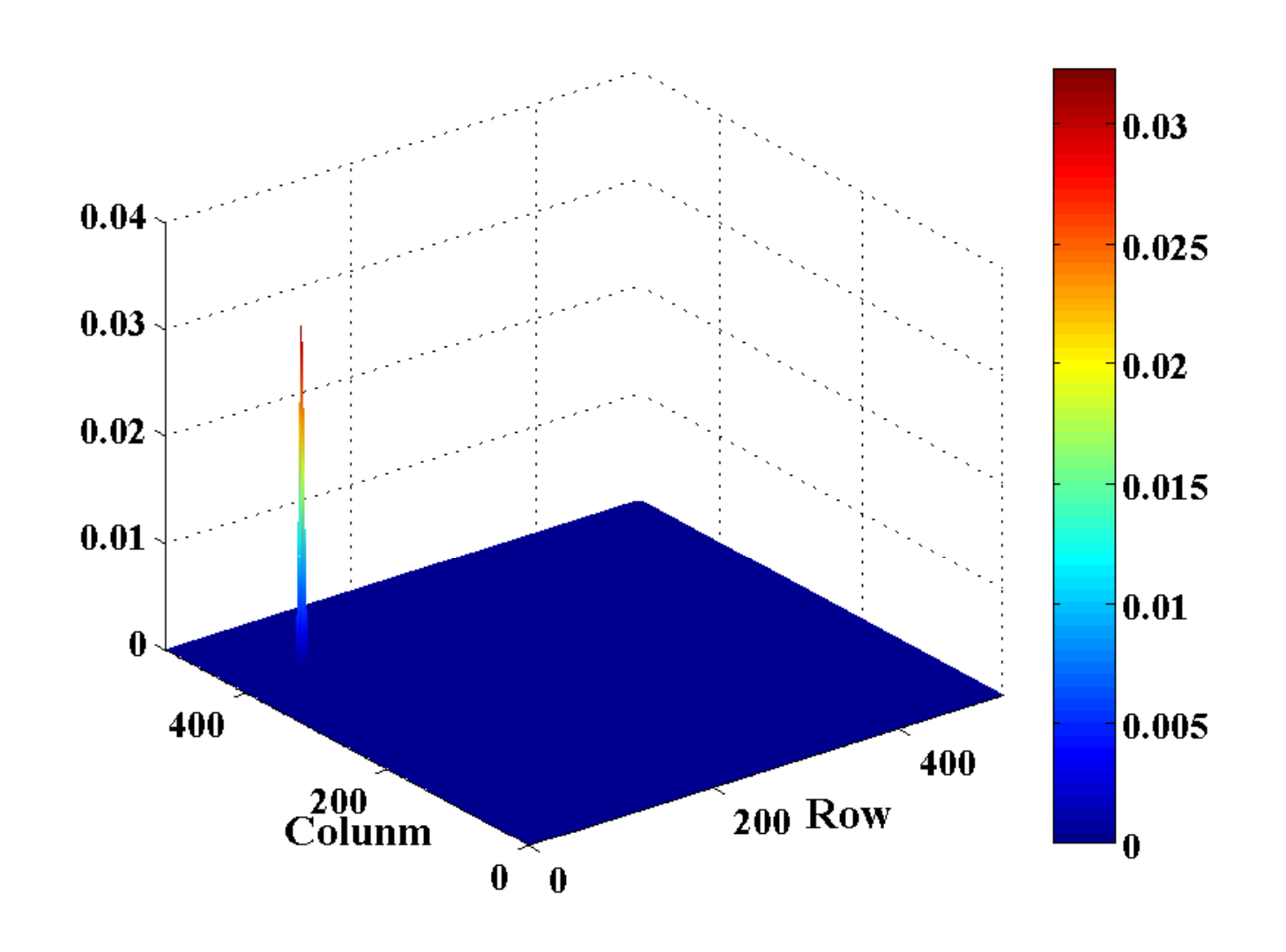}\label{1egn}}
 \subfigure[]
 {\includegraphics[height=5cm,width=5cm]{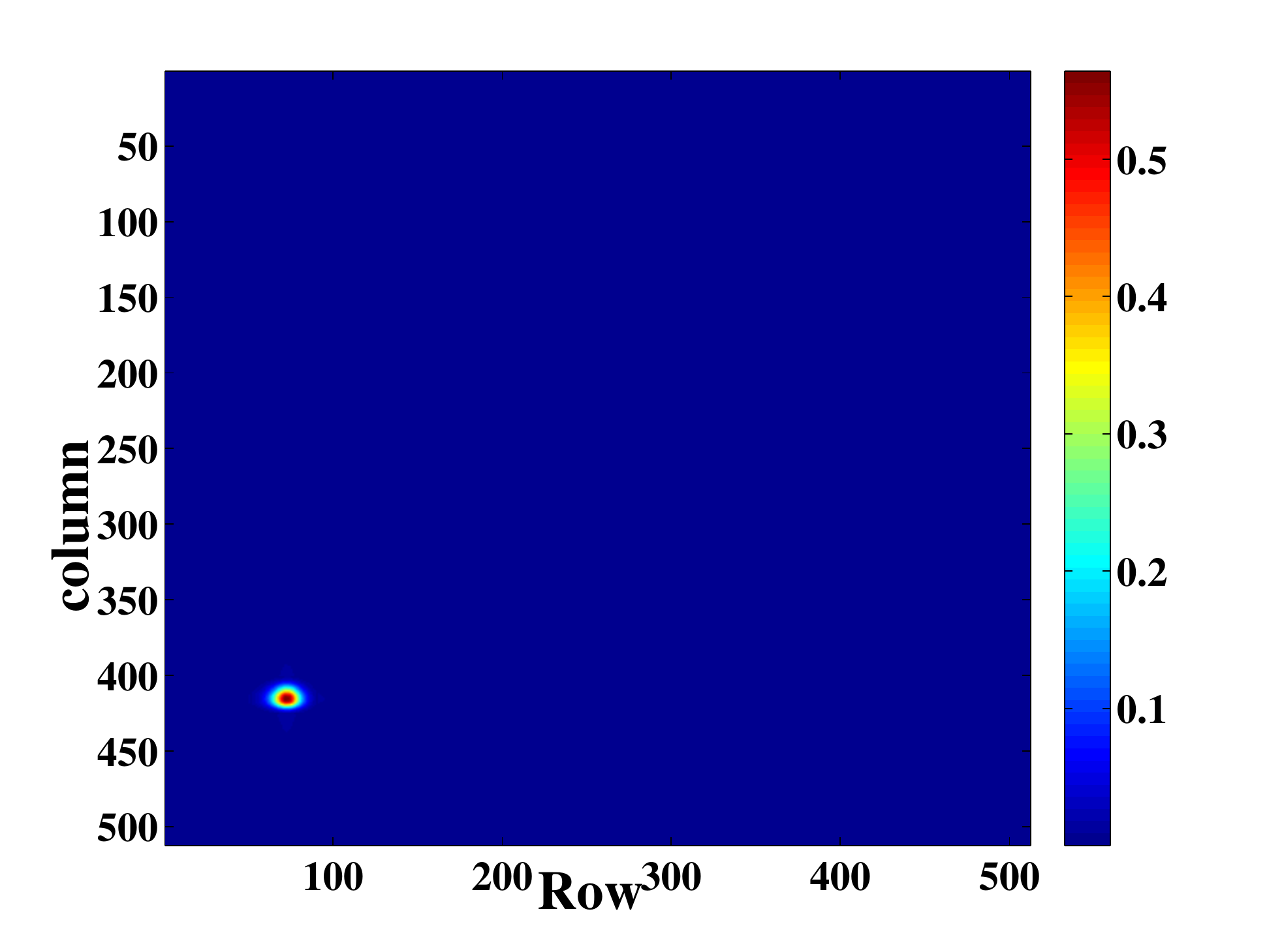}\label{1egnz}}
 \subfigure[]
 {\includegraphics[height=5cm,width=5cm]{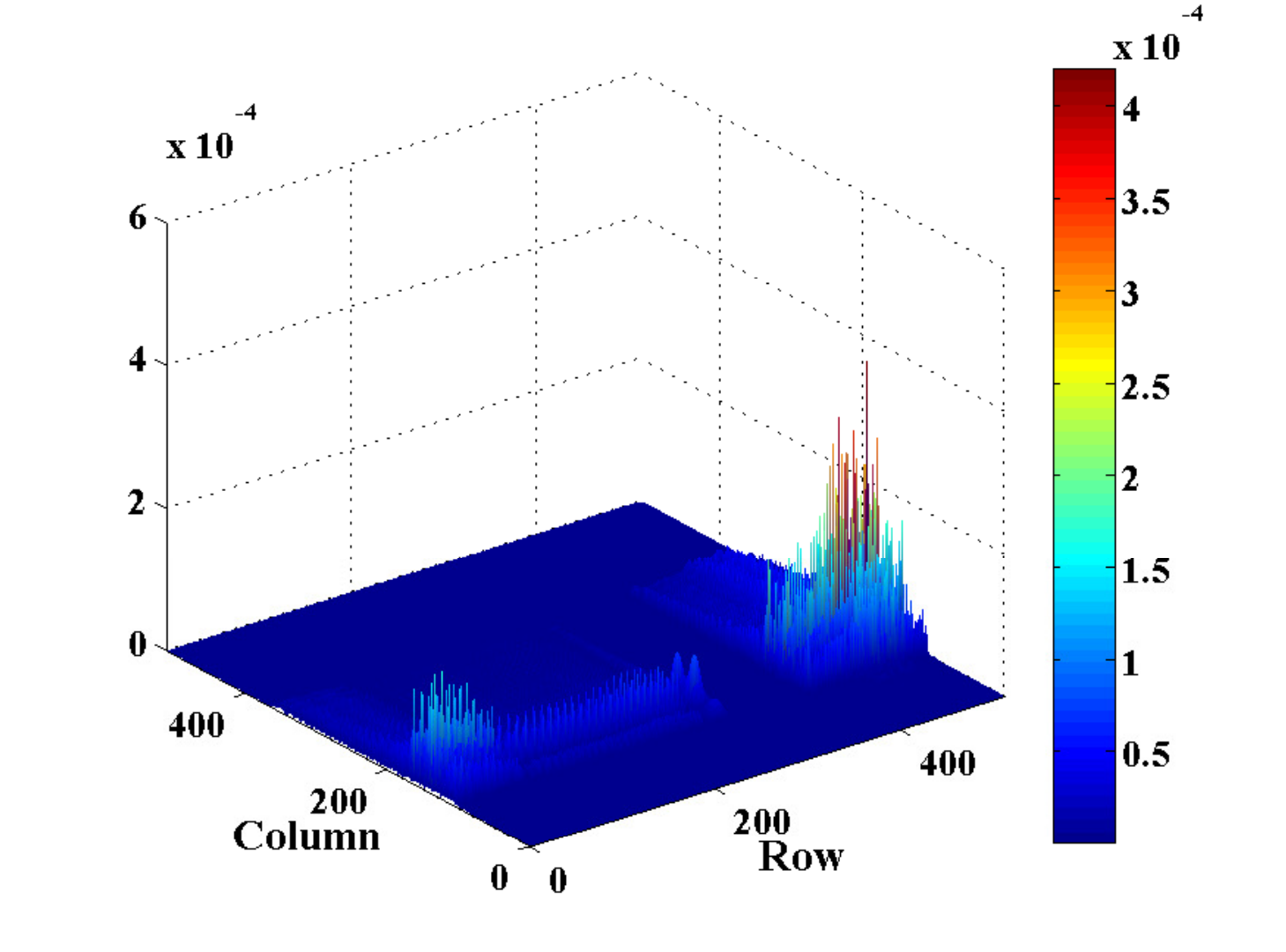}\label{512egn}}
 \subfigure[]
 {\includegraphics[height=5cm,width=5cm]{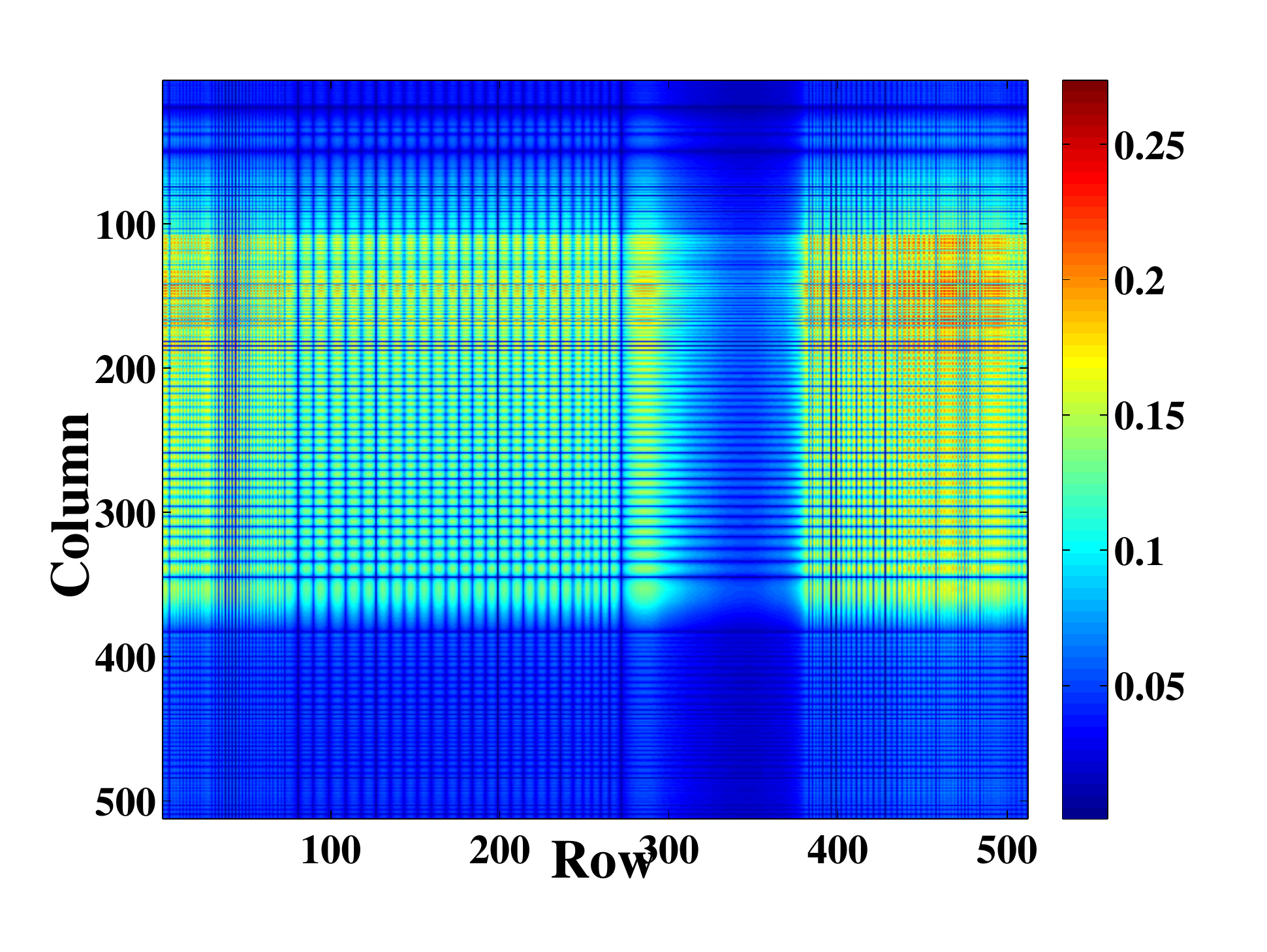}\label{512egnz}}

\caption{$(a) \text{ and }(b)$ First eigenfunction, $(c)\text{ and }(d)$ Last eigenfunction of Lena's image.}\label{loc_eig}
\end{figure}

\newpage

\section{Image denoising based on the SCSA approach}\label{sec6}
As described in the introduction, the novelty of this work is the way we select the set of adaptive functions for image representation and denoising. Comparing to the existing methods, the proposed method uses only negative eigenvalues and associated squared $L^2$-normalized eigenfunctions of the semi-classical Schr\"odinger operator, where the image is considered as a potential of this operator. This quantities which depend only on the image, represent the two main features which account for the performance improvement. It has been also shown in the previous section that good results are obtained with a small number of negative eigenvalues (i.e.; for $h$ large enough), this comes from the localization property of the eigenfunctions and the pertinent information that they contain. In fact the first eigenfunction gives a good localization of the largest
peak in the image, the second for the two peaks that follow the largest peak, then as the order of the eigenfunctions increases, the oscillations become more important (see figure $\ref{loc_eig}$) so they gives information of the smaller details in the image. This is a well-known results, indeed, in \cite{Pankov} we show
%
that the eigenfunctions corresponding to less significant eigenvalues are oscillating having asymptotically a sine behavior describing the details of the signal. If the signal is noisy, these eigenfunctions will describe the noise components. So removing these components helps to reduce the noise. However, because of the nonlinearity of the method, instead of a naive truncation, which may lead to loose information about the signal (since all the eigenfunctions contain information about the signal), an alternative consists in increasing the semi-classical parameter value leading to reduce the number of eigenfunctions and hence reduce the effect of the noise. 

We are going to show the efficiency and the stability of this method through some numerical results. The experiments have been carried out on 2D images which are standard testing images of most state-of-the-art denoising algorithms. The images are contaminated by additive Gaussian white noise with zero mean and different levels of standard deviation $\sigma$ (i.e.; different values of signal-to-noise ratio (SNR)), the noise is added using the command Matlab $imnoise$.

As a first step and by using only the visual performance, we will show through geometric image  that in the denoising process, the SCSA method preserves the edges even at high level of noise as illustrated in figures $\ref{dam7_5}$, $\ref{dam30}$ and $\ref{dam50}$.

\begin{figure}[!h]
 \centering
 \subfigure[]
 {\includegraphics[height=4cm,width=4cm]{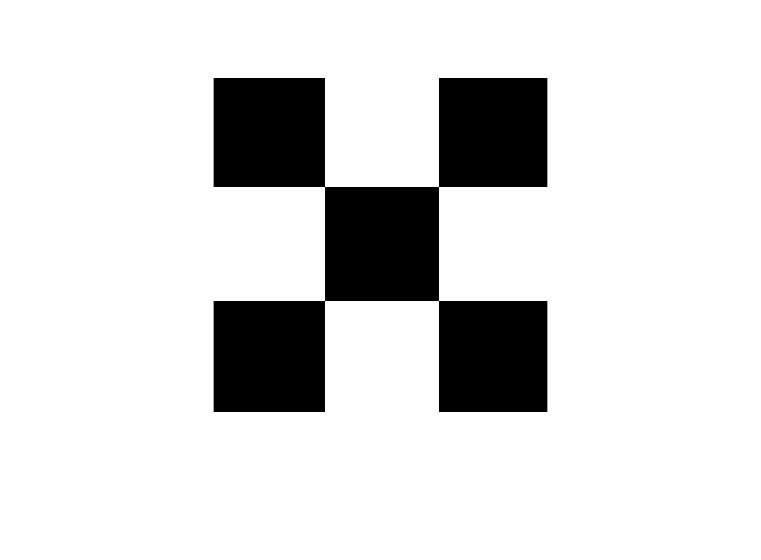}}
 \subfigure[]
 {\includegraphics[height=4cm,width=4cm]{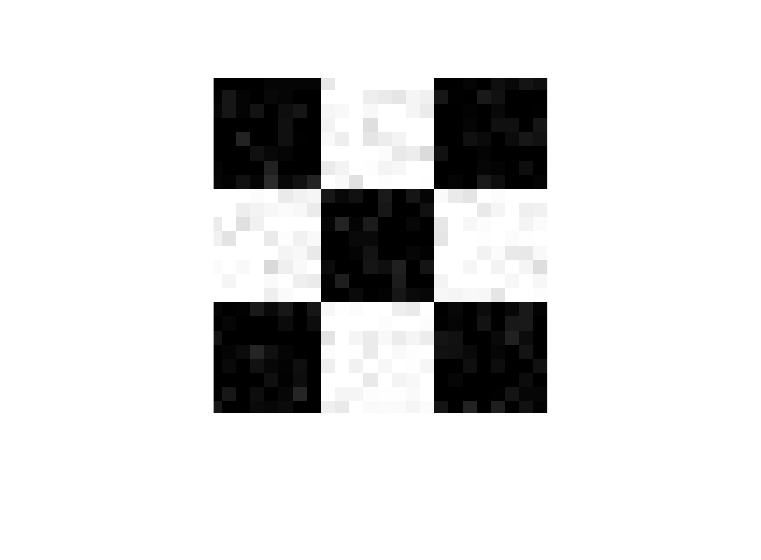}}
 \subfigure[]
 {\includegraphics[height=4cm,width=4cm]{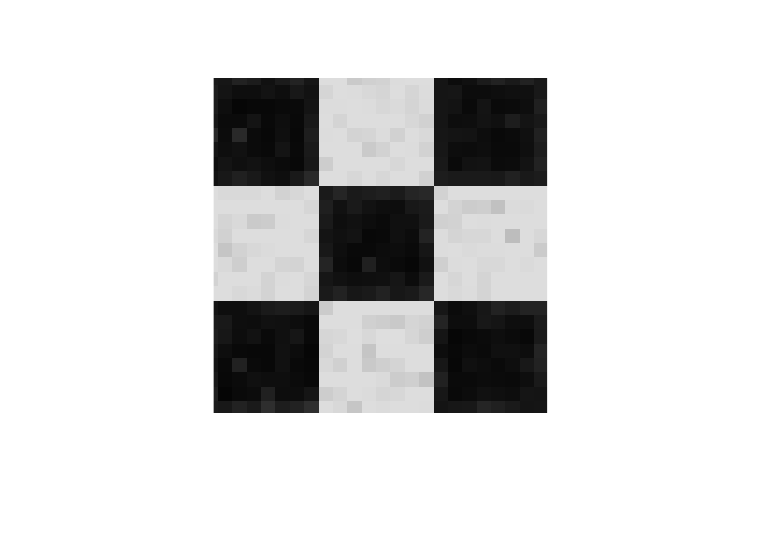}}
\caption{(a) Original image, (b) Noisy image ($\sigma=7.5$, SNR = $24.65$ dB), (c) Denoised image.}\label{dam7_5}
\end{figure}
\begin{figure}[!h]
 \centering
 \subfigure[]
 {\includegraphics[height=4cm,width=4cm]{DamierOriginal.eps}}
 \subfigure[]
 {\includegraphics[height=4cm,width=4cm]{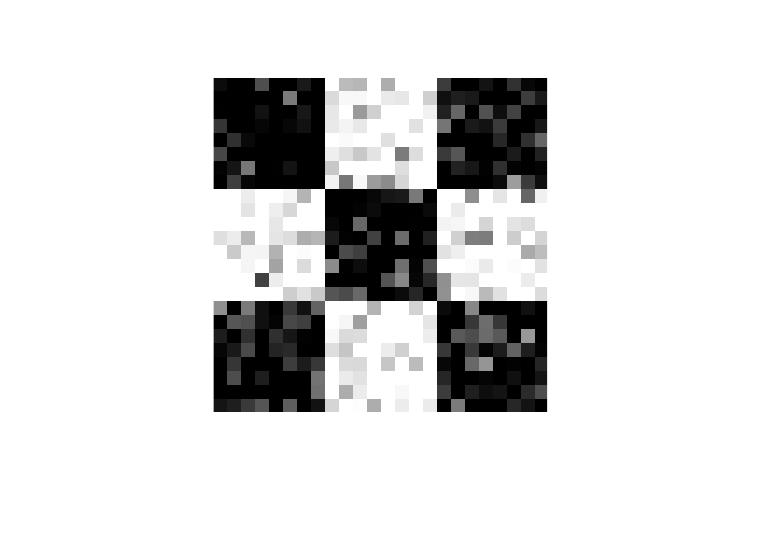}}
 \subfigure[]
 {\includegraphics[height=4cm,width=4cm]{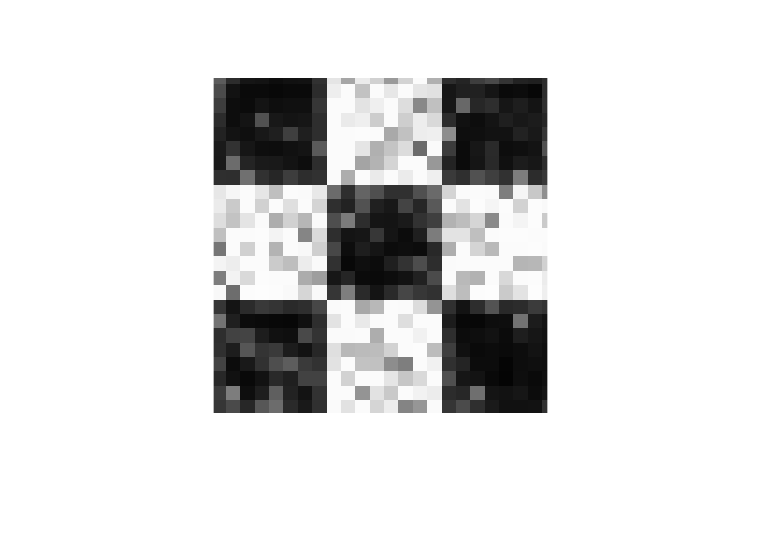}}
\caption{(a) Original image, (b) Noisy image ($\sigma=30$, SNR = $12.58$ dB), (c) Denoised image.}\label{dam30} 
\end{figure}
\begin{figure}[!h]
 \centering
 \subfigure[]
 {\includegraphics[height=4cm,width=4cm]{DamierOriginal.eps}}
 \subfigure[]
 {\includegraphics[height=4cm,width=4cm]{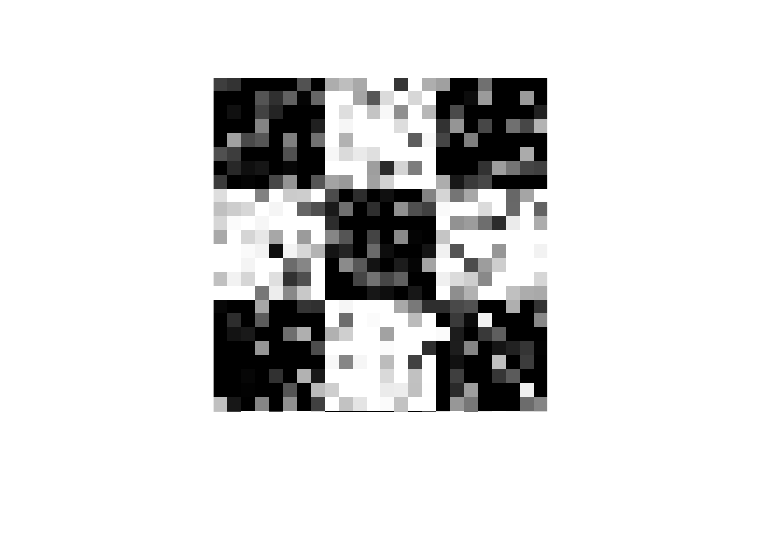}}
 \subfigure[]
 {\includegraphics[height=4cm,width=4cm]{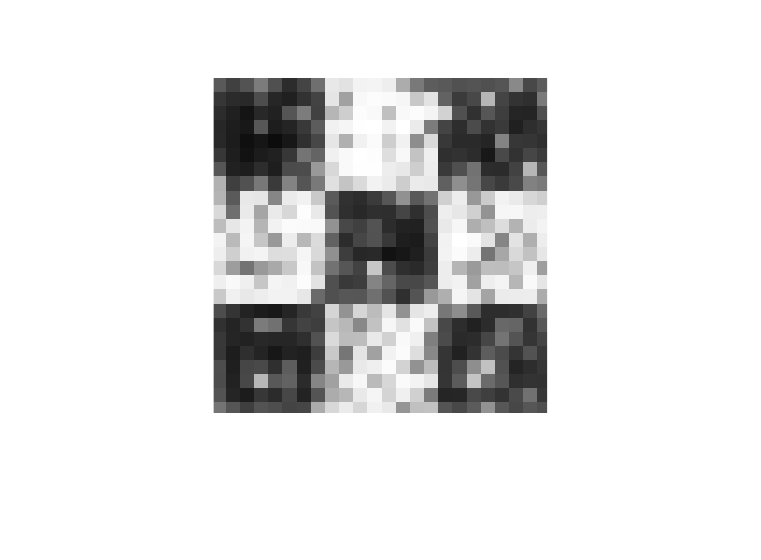}}
\caption{(a) Original image, (b) Noisy image ($\sigma=50$, SNR = $8.27$ dB), (c) Denoised image.}\label{dam50}
\end{figure}

In the following, some results obtained are presented. The objective performance measures used are defined as follows: let $I$ the noise-free image of size $N \times N$ and $\widetilde{I}$ be the denoised image.

\begin{itemize}[$\bullet$]
\item Mean Square Error
\begin{equation}
\text{MSE} = \frac{1}{N^2}\sum_{i=1}^{N}\sum_{j=1}^{N} \left(I[i,j]-\widetilde{I}[i,j]\right).
\end{equation}

\item Peak-Signal-to-Noise Ratio (PSNR)
\begin{equation}
\text{PSNR} = 10\,log_{10} \left(\frac{L^2}{MSE} \right),
\end{equation}
where $L$ is the dynamic values ​​of the pixels, which is $1$ or $255$ and MSE is the mean square error.

\item Mean Structural SIMilarity index (MSSIM)
\begin{equation}
\text{MSSIM} = \frac{1}{N} \sum_{i=1}^{N} \text{SSIM}(i,i),
\end{equation} 
\begin{equation}
\text{SSIM}(i,j) = \frac{(2\mu_i \mu_j+\epsilon_1)(2\sigma_{ij}+\epsilon_2)}{(\mu_i^2+\mu_j^2+\epsilon_1)(\sigma_i^2+\sigma_j^2+\epsilon_2)},
\end{equation}
where $\mu_i$ and $\mu_j$ are the standard deviation of $x$ and $y$, $\sigma_{ij}$ is the covariance between $i$ and $j$, $\epsilon_1$ and $\epsilon_2$ ensure the stability when either $(\mu_i^2+\mu_j^2)$ or $(\sigma_i^2+\sigma_j^2)$ is close to zero. The SSIM is defined over a local window centered at $[i,j]$ and an average over such windows gives a single measure for the entire image, named as Mean SSIM (MSSIM) \cite{Wang2004}.

\end{itemize}

Figures $\ref{L}$ and $\ref{L10noise}$ show the original image of Lena and the noisy one. The standard deviation $\sigma$ is equal to $75$ and the corersponding SNR is equal to $11.24$ dB. 
\begin{figure}[!h]
 \centering
 \subfigure[]
 {\includegraphics[height=4.25cm,width=4.25cm]{OriginalLena.eps}\label{L}}
 \subfigure[]
 {\includegraphics[height=4.25cm,width=4.25cm]{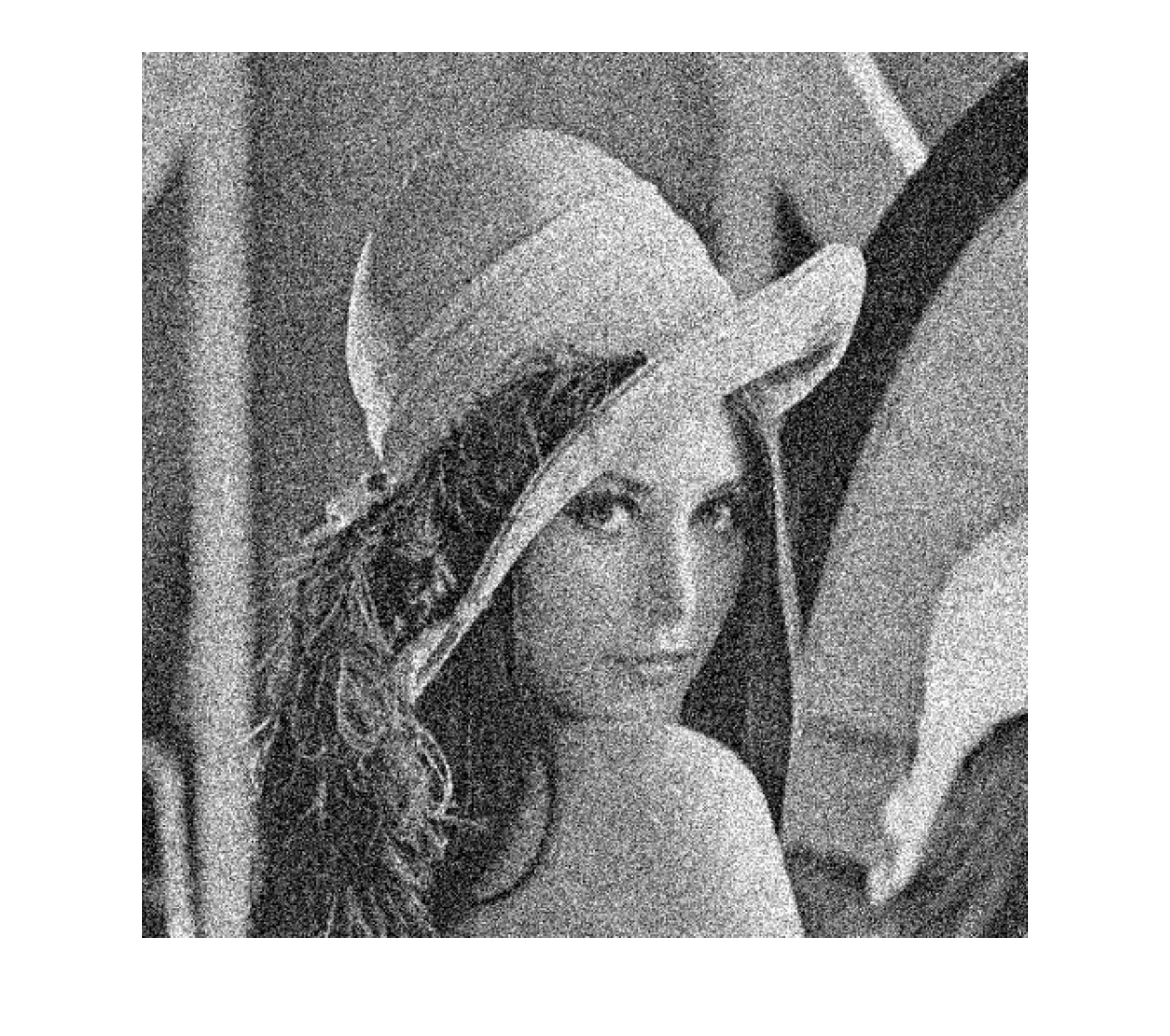}\label{L10noise}}
\caption{(a) Original image of Lena, (b) Noisy image for $\sigma=75$ (SNR = $11.24$ dB).} 
\end{figure}

Figures $\ref{Leqm75}$, $\ref{Lpsnr75}$ and $\ref{Lmssim75}$ show that the optimal value of $h$ is equal to $1.65$. Figure $\ref{L75D165}$ illustrates the denoising of the Lena's image using the optimal value of $h$. However, the use of an $h$ smaller than the optimal value, does not filter completely the noise, but helps to reconstruct the noisy image (see figure $\ref{L75D12}$).

\begin{figure}[!h]
 \centering
 \subfigure[]
 {\includegraphics[height=3.75cm,width=3.75cm]{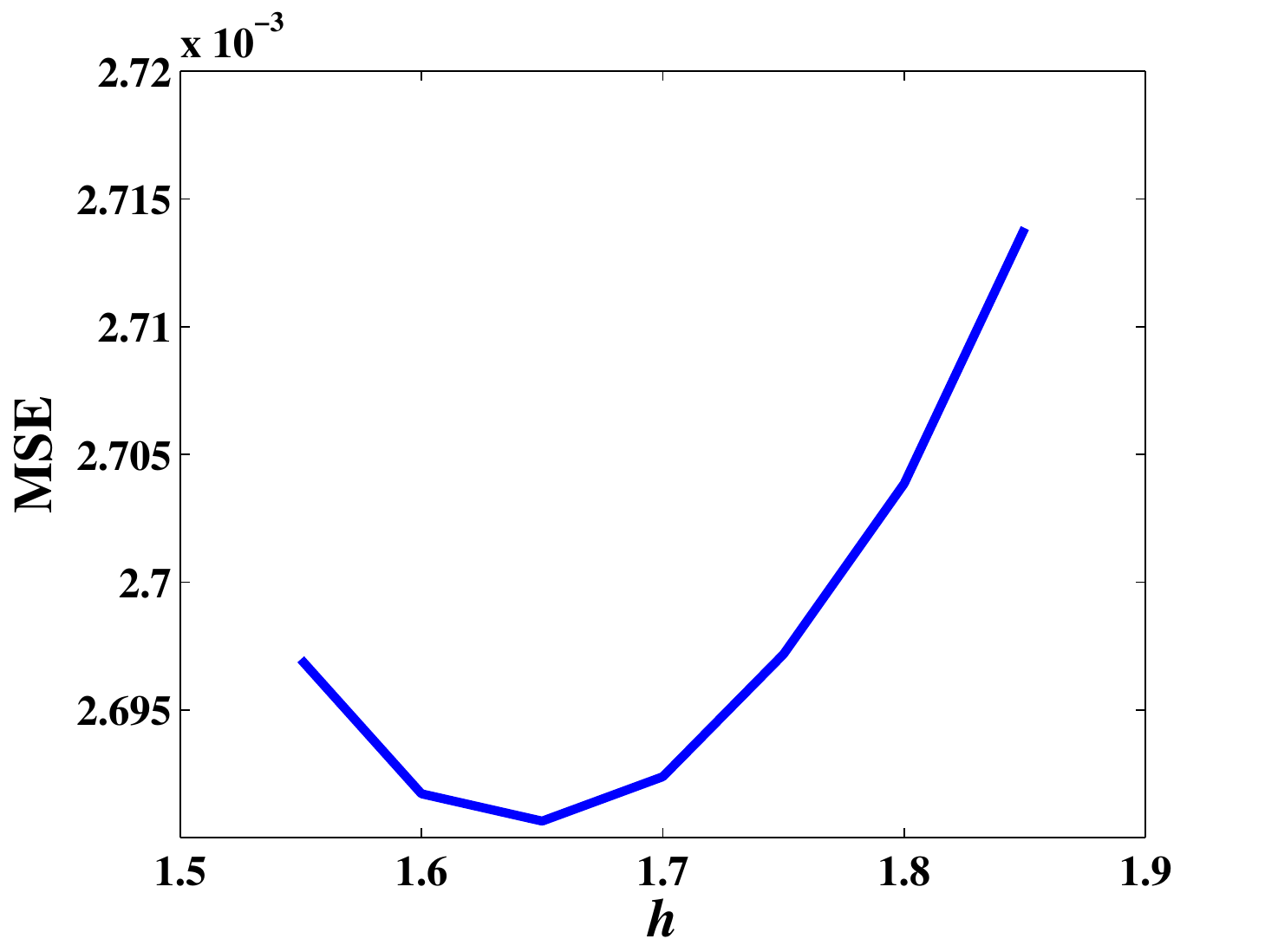}\label{Leqm75}}
 \subfigure[]
 {\includegraphics[height=3.75cm,width=3.75cm]{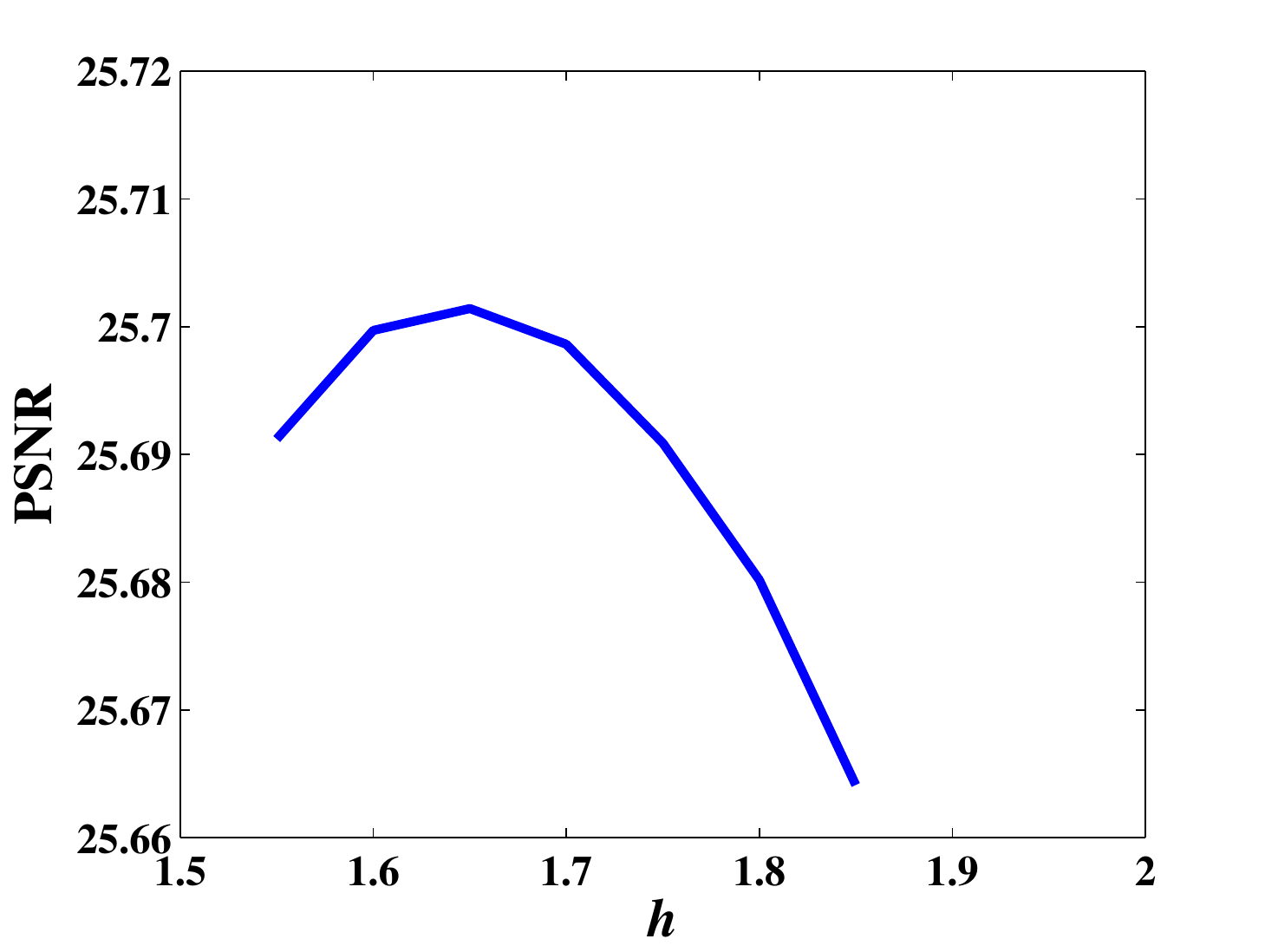}\label{Lpsnr75}}
  \subfigure[]
 {\includegraphics[height=3.75cm,width=3.75cm]{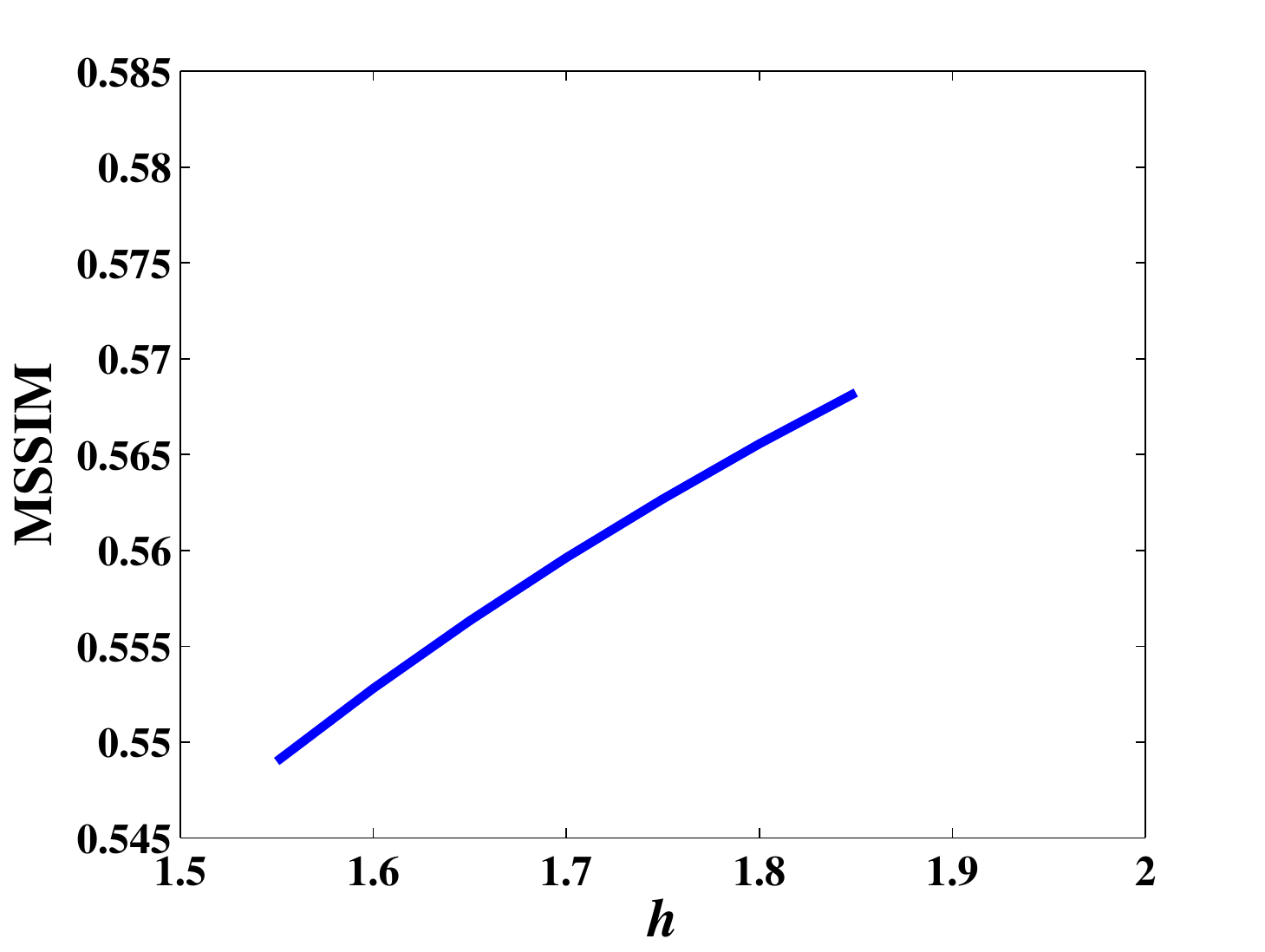}\label{Lmssim75}}
\caption{Mean square error (Fig. (a)), PSNR (Fig. (b)) and MSSIM (Fig. (c)) for different value of $h$.}
\end{figure}

\begin{figure}[!h]
 \centering
 \subfigure[]
 {\includegraphics[height=3.75cm,width=3.75cm]{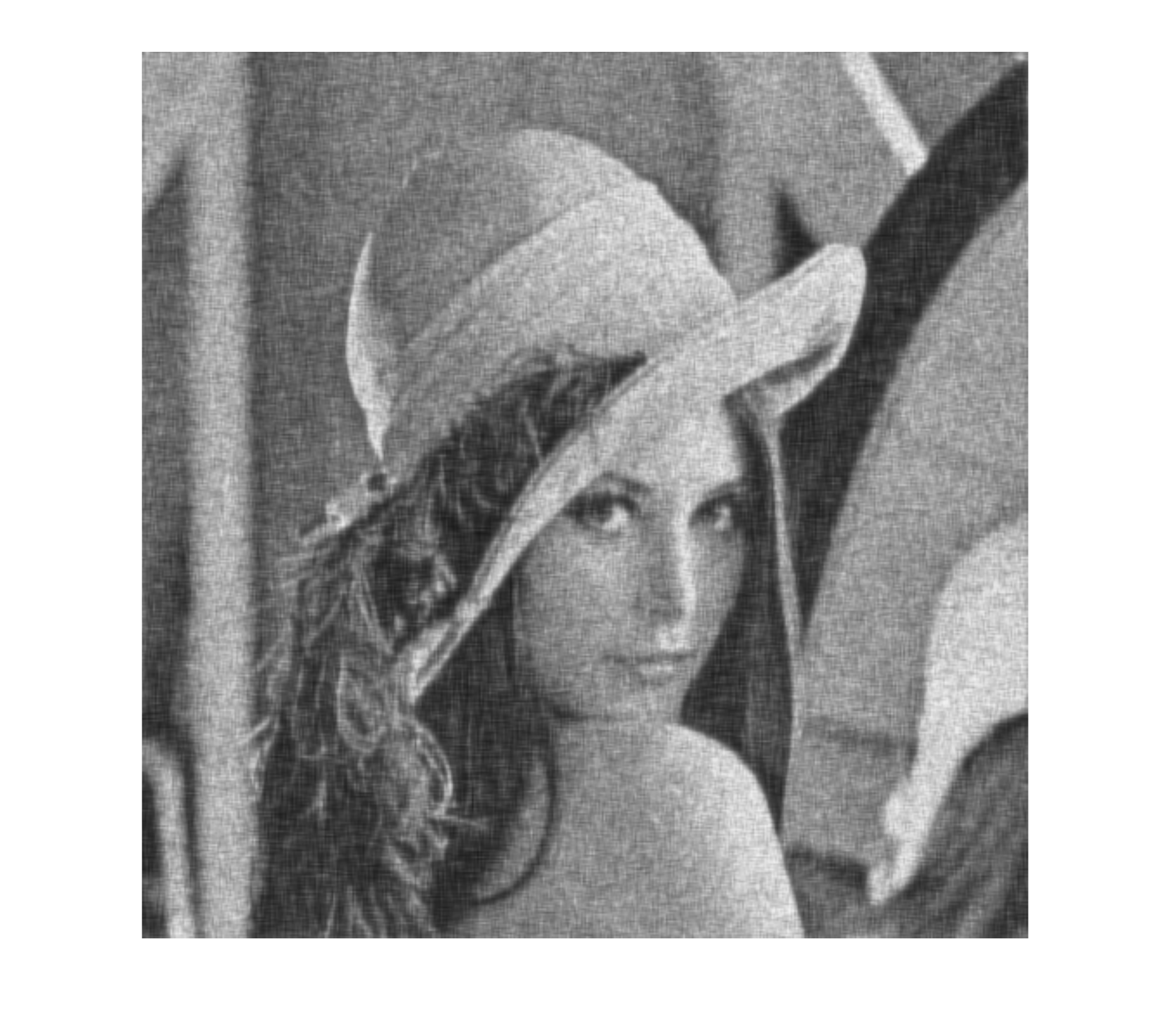}\label{L75D12}}
  \subfigure[]
 {\includegraphics[height=3.75cm,width=3.75cm]{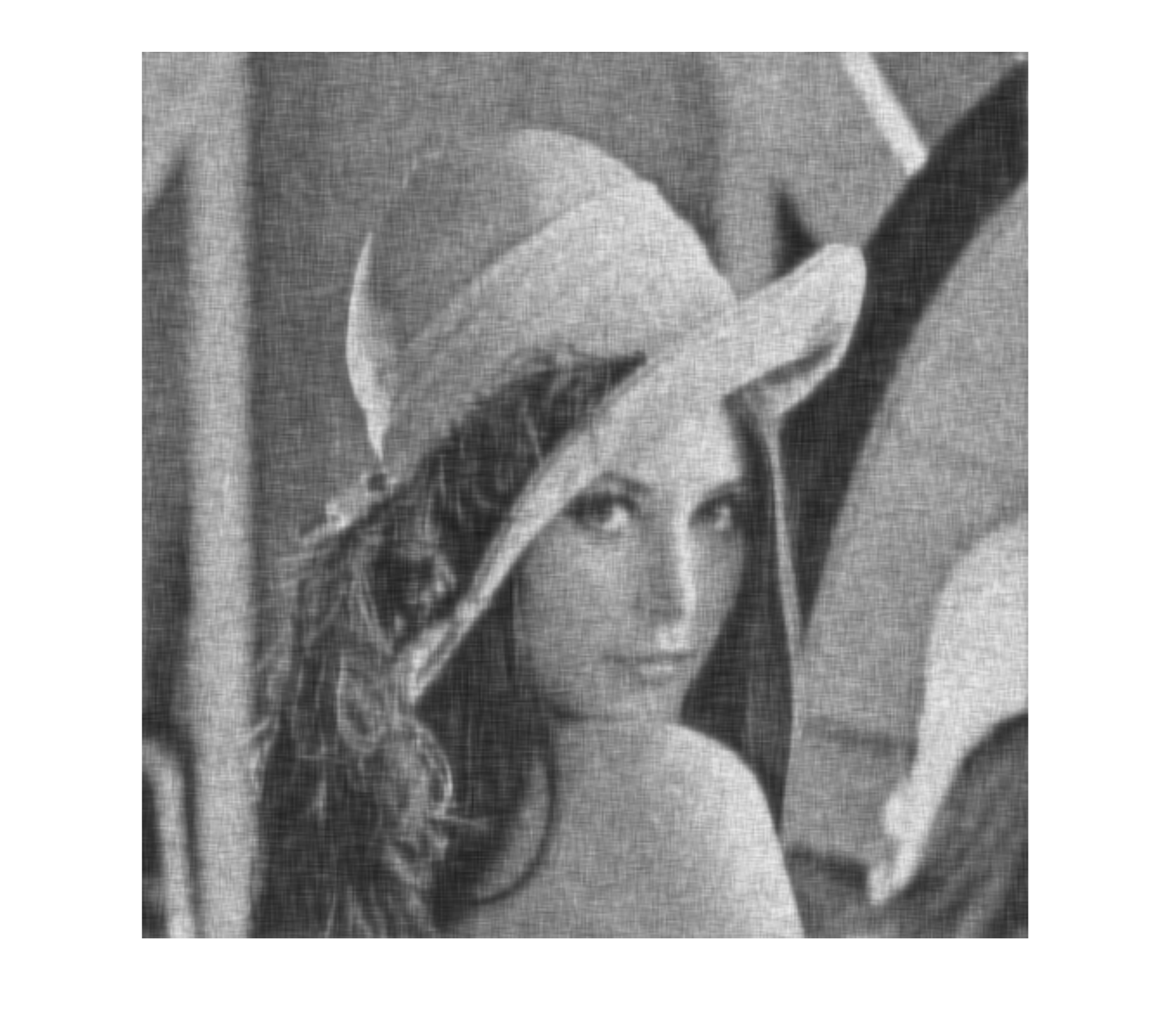}\label{L75D165}}
 \subfigure[]
 {\includegraphics[height=3.75cm,width=3.75cm]{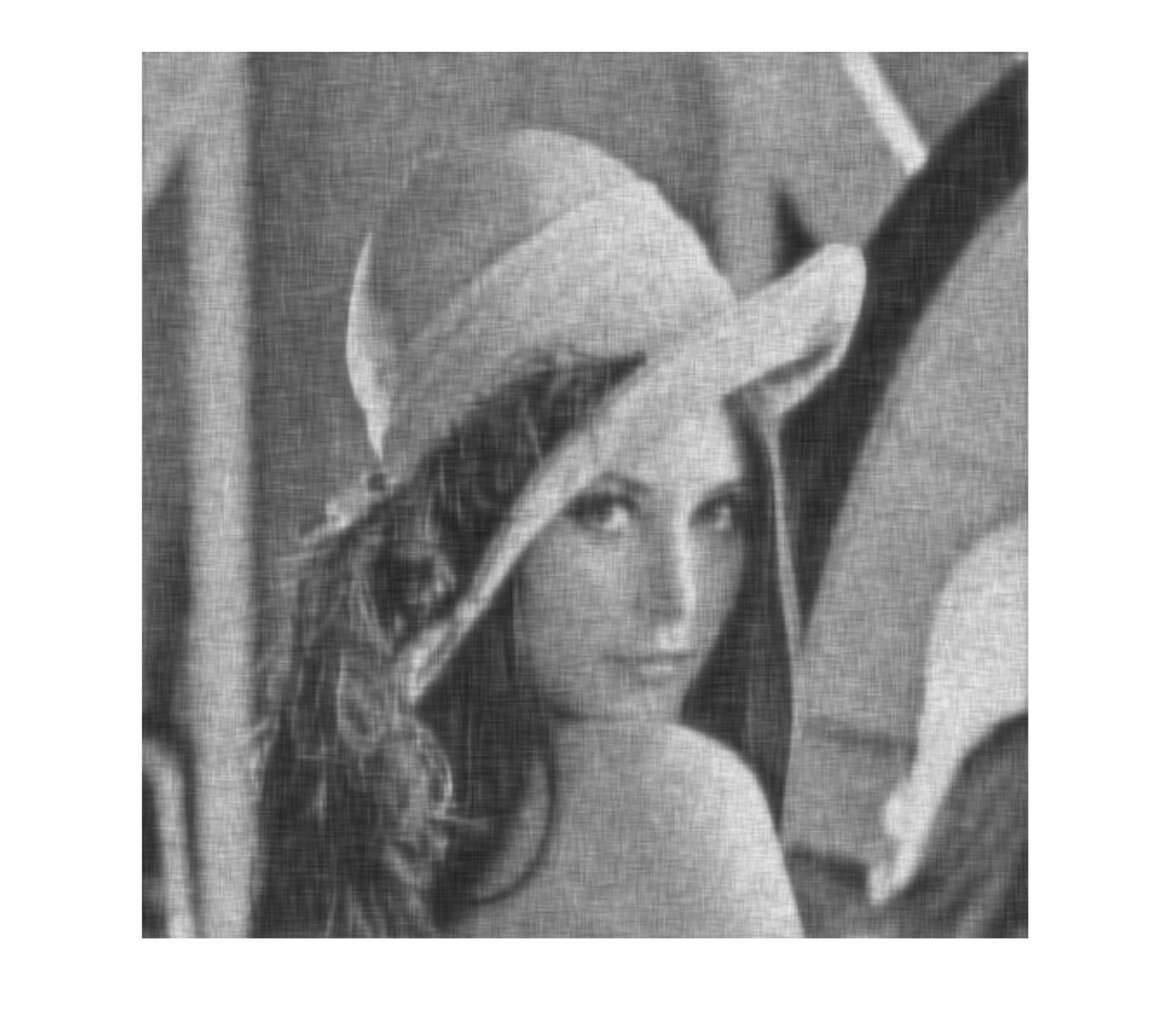}\label{L75D2}}
\caption{(a) Denoised image with $h = 1.2$ (MSE = $0.0029$, PSNR = $25.4050$ dB, MSSIM = $0.5130$), (b) Denoised image with $h = 1.65$ (MSE = $0.0027$, PSNR = $25.7015$ dB, MSSIM = $0.5563$), (c) Denoised image with $h = 2$ (MSE = $0.0028$, PSNR = $25.6006$ dB, MSSIM = $0.5754$).}
\end{figure}

Figure $\ref{Zoomedge}$ shows a zoom of the denoising result of Lena image obtained by the proposed method for $\sigma=75$ and $h=1.65$. We can see that the proposed method provides better visual quality and the edges and textures of the image are better preserved.

\begin{figure}[!h]
 \centering
 \subfigure[]
 {\includegraphics[height=4cm,width=4cm]{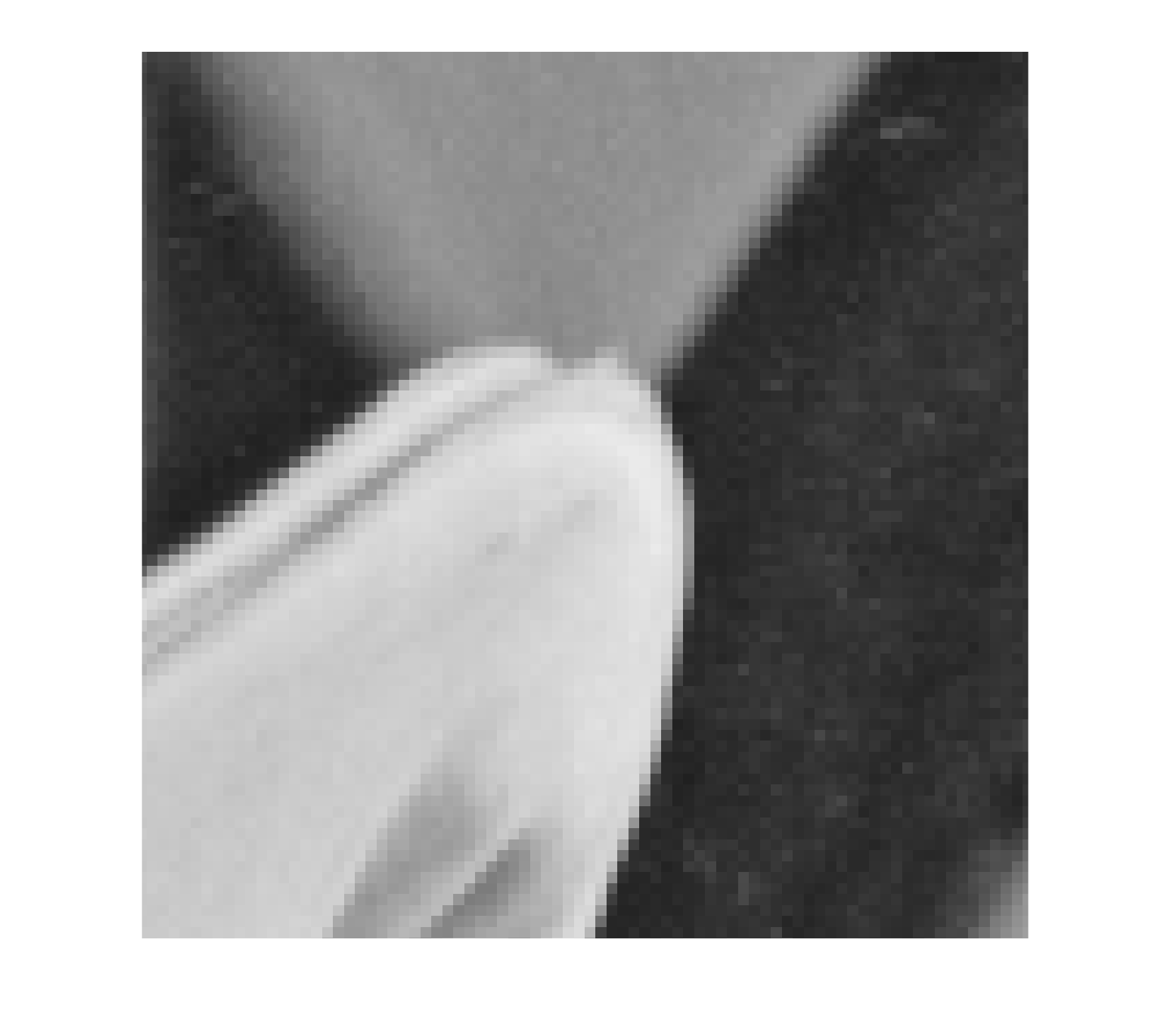}}
 \subfigure[]
 {\includegraphics[height=4cm,width=4cm]{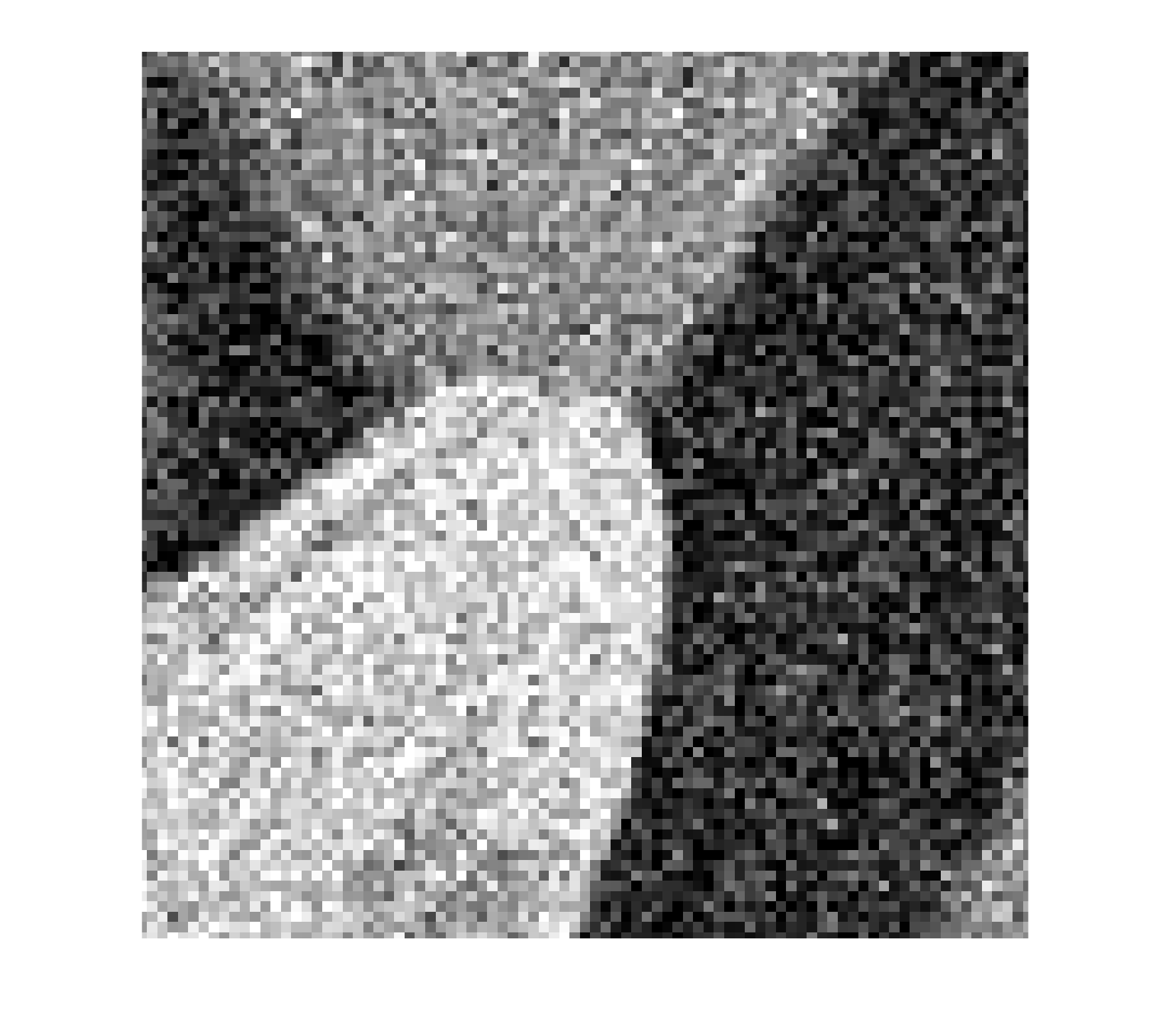}}
 \subfigure[]
 {\includegraphics[height=4cm,width=4cm]{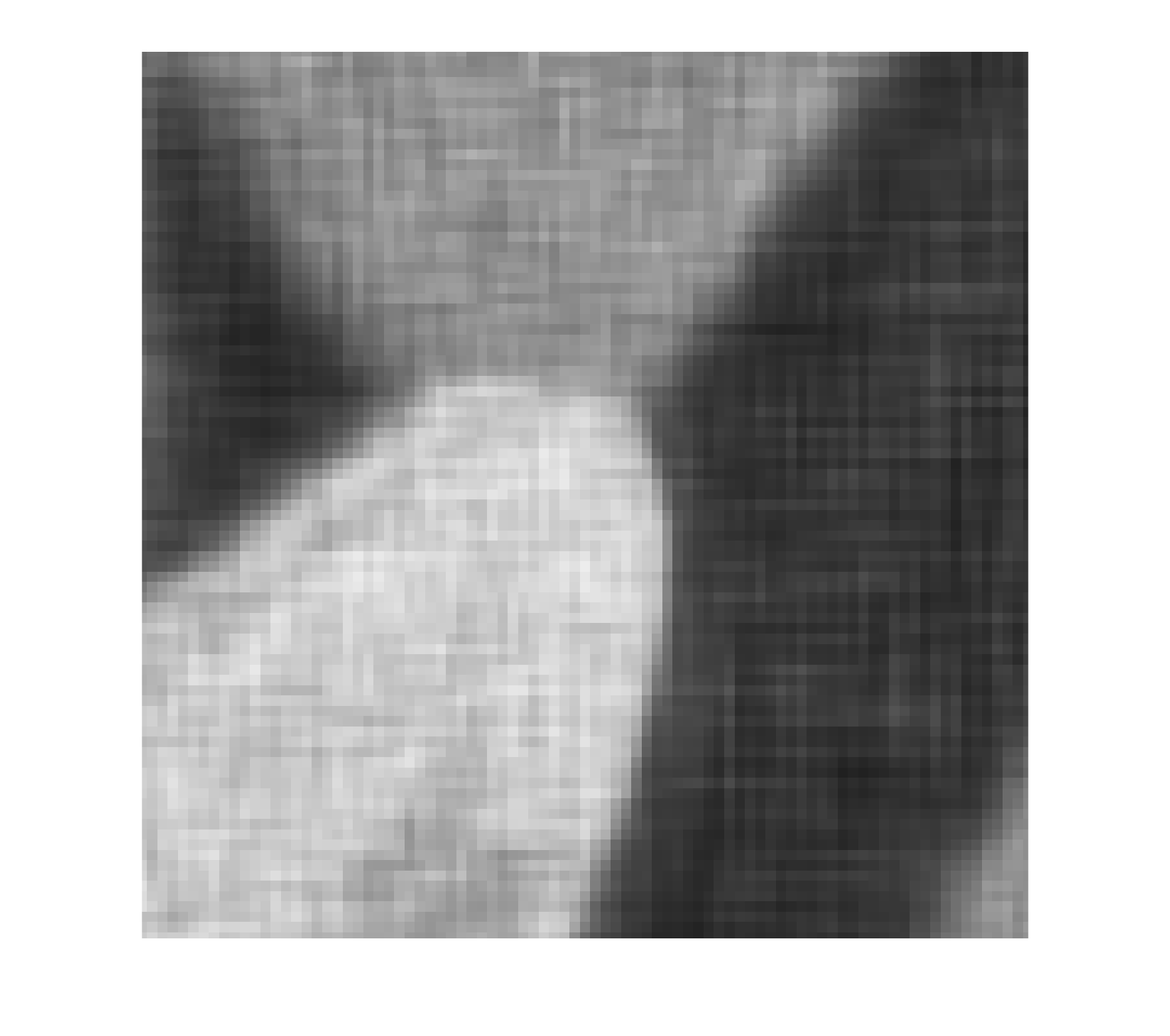}}
\caption{Zoom on denoising of Lena corrupted by noise with $\sigma = 75$, (a) Original image, (b) Noisy image, (c) Denoised image.}
\label{Zoomedge}
\end{figure}

The histogram of the original image of Lena, the noisy $(\sigma=75)$ and the denoised one are illustrated, respectively, in figure $\ref{Histo}$. Figure $\ref{Hnoise}$, which represents the histogram of the noisy image, has the shape of the Gaussian function. Using the SCSA method in the denoising process (figure \ref{Hdenoised}), the shape of the original image (figure \ref{Horiginal}) is obtained even at high level of noise. 

\begin{figure}[!h]
 \centering
 \subfigure[]
 {\includegraphics[height=4cm,width=4cm]{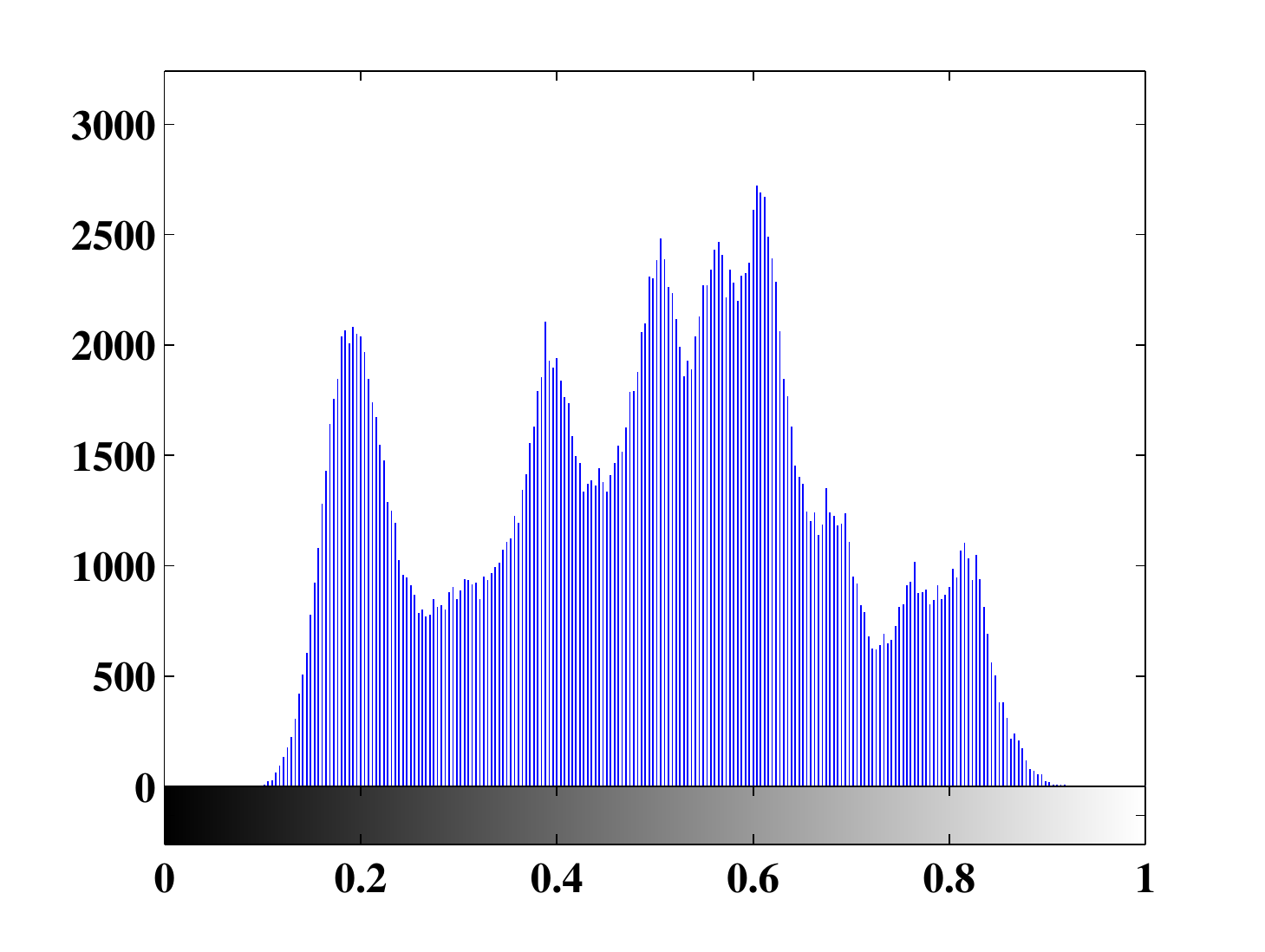}\label{Horiginal}}
 \subfigure[]
 {\includegraphics[height=4cm,width=4cm]{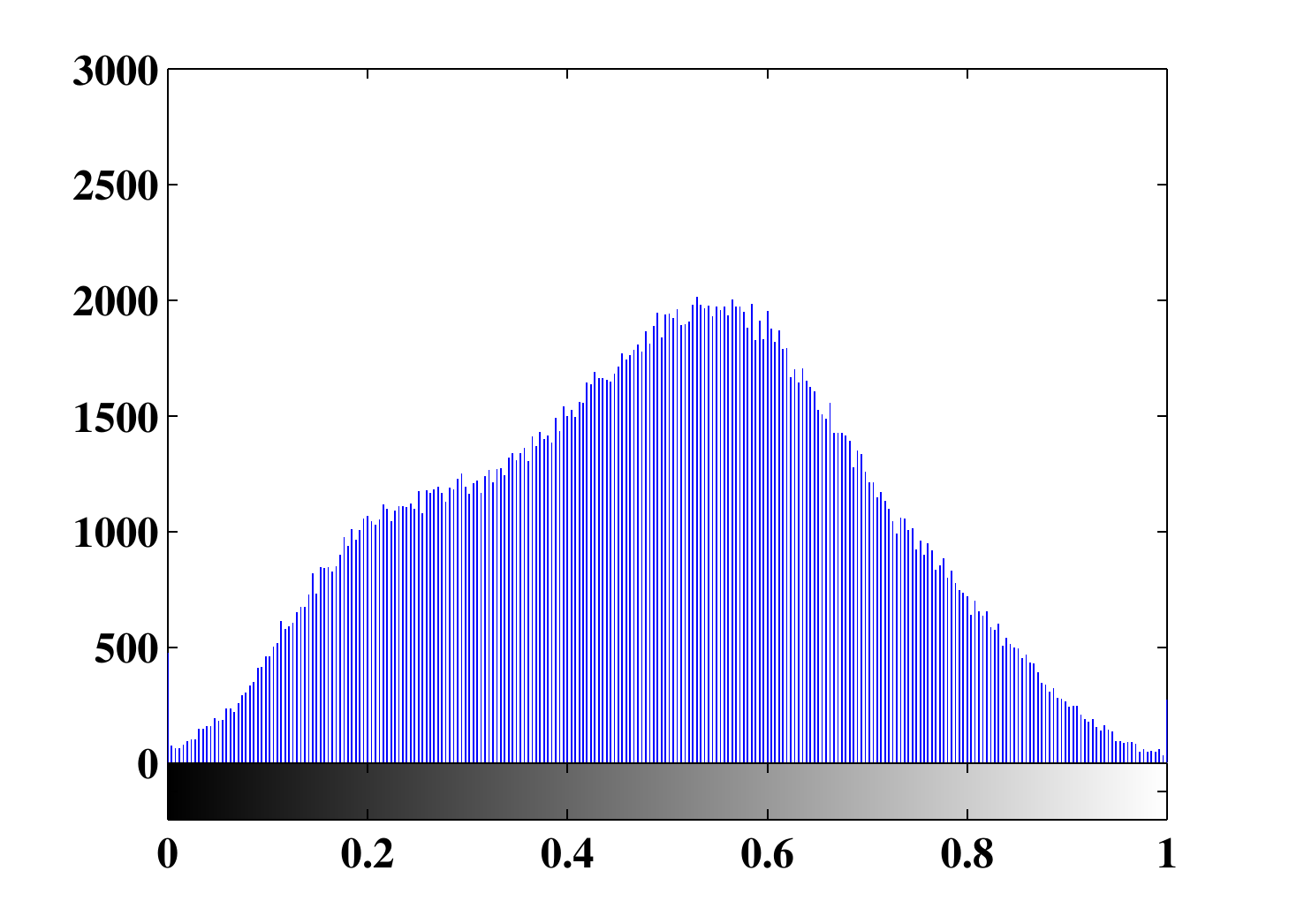}\label{Hnoise}}
 \subfigure[]
 {\includegraphics[height=4cm,width=4cm]{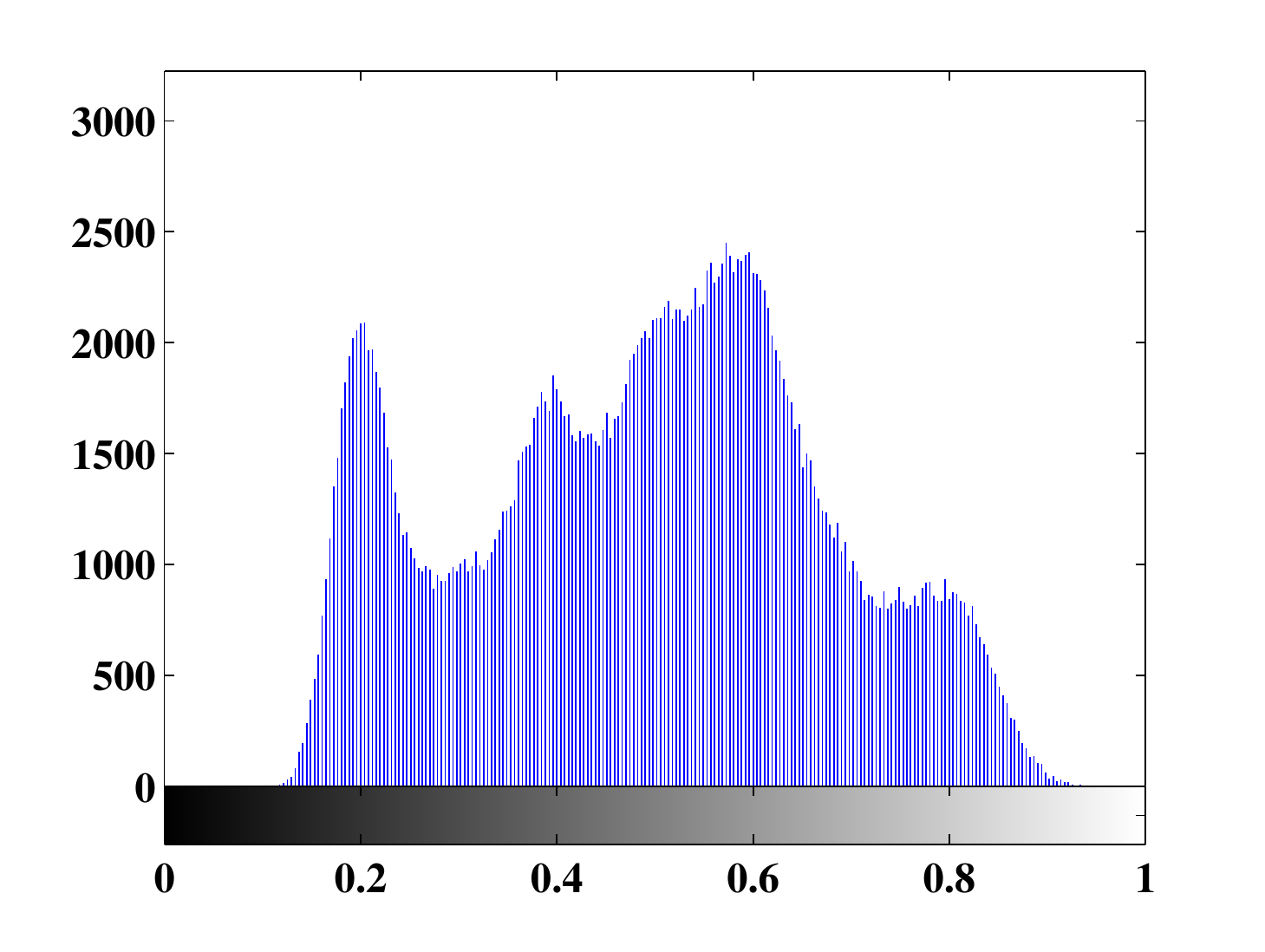}\label{Hdenoised}}
\caption{The histogram of, (a) Original image of Lena, (b) Noisy image $(\sigma = 75)$, (c) Denoised image.}
\label{Histo}
\end{figure}

The quantitative first results of the proposed algorithm are compared to the state-of-the-art models which are Total Variation $\cite{Elad2006, Wang2013, Zhang2012}$ and K-SVD \cite{Aharon2006,He2011,Yan2014}. All the parameters of these methods are set as what have been suggested to be the optimal one in the original paper. For our approach, the optimal values of $\lambda$ and $\gamma$ are 0 and 4, then for this values, the optimal value of $h$ has been chosen such that we use only the most significant eigenfunctions and associated negative eigenvalues. Table \ref{T1} shows the obtened results for Lena's image.

\begin{table}[!h]
\centering
\begin{tabular}{l c c c c c c c c c c c c c }
\hline \vspace{1pt}
   Noise variance & & \multicolumn{3}{c}{Proposed method} & & \multicolumn{2}{c}{ TV } & & \multicolumn{2}{c}{K-SVD} \\ 
\cline{3-5}  
\cline{7-8} 
\cline{10-11}
 \vspace{1pt}
 $\quad \quad \sigma$ & & $h$ & PSNR & SSIM & & PSNR & SSIM & & PSNR & SSIM\\ 
\hline 

 $\quad \quad  20$   & & 0.550 & 32.1 &  0.995 & & 31.3 & 0.843 & & 29.7 & 0.857 \\
 $\quad \quad  30$   & & 0.800 & 30.1 &  0.991 & & 29.6 & 0.809 & & 27.8 & 0.805 \\
 $\quad \quad  40$   & & 0.975 & 28.7 &  0.987 & & 28.3 & 0.779 & & 26.2 & 0.750 \\
 $\quad \quad  50$   & & 1.000 & 27.5 &  0.983 & & 27.2 & 0.756 & & 25.0 & 0.747 \\
 $\quad \,\,\,\, 100$ & & 2.100 & 24.4 &  0.971 & & 22.8 & 0.678 & & 21.5 & 0.559 \\
  \hline
\end{tabular}
\caption{The PSNR (in dB) and SSIM results of the denoised images at different noise levels by  TV, K-SVD, and proposed method.}
\label{T1}
\end{table}

\newpage
\section{Discussion and conclusion}\label{sec7}

A new image representation and analysis method has been proposed in this paper inspired from semi-classical results of the Schr\"odinger operator. The image is represented using spatially shifted and localized functions that are given by the squared $L^2$-normalized eigenfunctions of the Schr\"odinger operator associated to negative eigenvalues. We have shown that this approximation becomes exact when the semi-classical parameter $h$ converges to zero. However the number of eigenfunctions increases when $h$ decreases, so we have shown through some numerical results that a relatively small number of eigenfunctions (large enough $h$) is enough to reconstruct the image which makes this method very interesting for image processing applications like coding. 

Moreover, thanks to its interesting properties, this method seems to be also useful for image denoising. The main idea is to choose an appropriate value for the semi-classical parameter $h$ to filter the noise. The denoising property of the SCSA is under consideration along with the comparison of the SCSA to standard image processing methods.


\section*{Acknowledgments} This work was conducted when the first author was visiting the Estimation, Modeling and ANalysis Group at the Computer, Electrical and Mathematical Sciences and Engineering (CEMSE) division at King Abdullah University of Science and Technology (KAUST). She would like to thank KAUST for its support and generous hospitality.

\end{document}